\documentclass{amsart}
\usepackage[pdfauthor={Joel David Hamkins, Russell Miller, and Kameryn J. Williams},
    pdftitle={Forcing as a computational process},
    hidelinks 
]{hyperref}

\newcommand{\xcheck}{\check{x}}
\newcommand{\ycheck}{\check{y}}
\newcommand{\avec}{\vec{a}}
\newcommand{\xvec}{\vec{x}}
\newcommand{\yvec}{\vec{y}}

\title{Forcing as a computational process}

\author{Joel David Hamkins}
  \address[Joel David Hamkins]
{O'Hara Professor of Logic, University of Notre Dame, 100 Malloy Hall, Notre Dame, IN 46556 USA \&\ Associate Faculty Member, Philosophy, University of Oxford, UK}
  \email{jdhamkins@nd.edu}
  \urladdr{http://jdh.hamkins.org}

\author{Russell Miller}
  \address[Russell Miller]
          {The Graduate Center of CUNY, Ph.D.\ Programs in Mathematics \& Computer Science, 365 Fifth Avenue, New York, NY 10016, USA
           \& Queens College of CUNY, Mathematics Dept., 65-30 Kissena Blvd., Flushing, NY 11367, USA}
  \email{Russell.Miller@qc.cuny.edu}
  \urladdr{http://qcpages.qc.cuny.edu/\~{ }rmiller}

\author{Kameryn J. Williams}
\address[Kameryn J. Williams]{Bard College at Simon's Rock \\
84 Alford Rd \\
Great Barrington, MA 01230, USA}
\email{kwilliams@simons-rock.edu}
\urladdr{http://kamerynjw.net}

\thanks{We thank the anonymous referee for their helpful comments.}
\thanks{The second author was supported by NSF grant \# DMS-1362206, Simons Foundation grant \# 581896, and several PSC-CUNY research awards.}
\thanks{Commentary can be made about this article on the first author's blog at \href{http://jdh.hamkins.org/}{http://jdh.hamkins.org/forcing-as-a-computational-process}.}

\usepackage{latexsym,amsfonts,amsmath,amssymb,mathrsfs}
\usepackage{multicol}
\usepackage{bbm}
\usepackage[hidelinks]{hyperref}
\usepackage{relsize}
\usepackage[rgb,dvipsnames]{xcolor}
\usepackage{tikz}
\usepackage{enumitem}

\newtheorem*{theorem*}{Theorem}

\newtheorem{theorem}{Theorem}
\newtheorem{maintheorem}[theorem]{Main Theorem}

\newtheorem*{maintheorem*}{Main Theorem}
\newtheorem*{maintheorems*}{Main Theorems}
\newtheorem{corollary}[theorem]{Corollary}
\newtheorem*{corollary*}{Corollary}
\newtheorem*{corollaries*}{Corollaries}

\newtheorem{lemma}[theorem]{Lemma}

\newtheorem{question}[theorem]{Question}
\newtheorem*{question*}{Question}

\newtheorem*{questions*}{Questions}
\newtheorem*{mainquestion*}{Main Question} 
\newtheorem*{openquestion*}{Open Question} 

\newtheorem{proposition}[theorem]{Proposition}

\newtheorem{definition}[theorem]{Definition}

\newcommand{\QED}{\end{proof}}

\def\proclaim[#1]{{\bf #1}}
\def\BF#1.{{\bf #1.}}

\def\says#1:#2\par{\item[#1] #2\par}

\newcommand{\Levy}{L\'{e}vy}

\newcommand{\N}{{\mathbb N}}
\renewcommand{\P}{{\mathbb P}}
\newcommand{\Q}{{\mathbb Q}}

\newcommand{\R}{{\mathbb R}}

\newcommand{\Fbar}{{\overline{F}}}

\newcommand{\one}{\mathbbm{1}} 
\makeatletter
\newcommand{\dotminus}{\mathbin{\text{\@dotminus}}}
\newcommand{\@dotminus}{%
  \ooalign{\hidewidth\raise1ex\hbox{.}\hidewidth\cr$\m@th-$\cr}%
}
\makeatother

\newcommand{\of}{\subseteq}

\newcommand{\set}[1]{\{\,{#1}\,\}}

\newcommand{\elesub}{\prec}

\newcommand{\dom}{\mathop{\rm dom}}

\newcommand{\Add}{\mathop{\rm Add}}

\newcommand{\image}{\mathbin{\hbox{\tt\char'42}}}

\newcommand{\satisfies}{\models}
\newcommand{\forces}{\Vdash}
\newcommand{\proves}{\vdash}

\newcommand{\smalllt}{\mathrel{\mathchoice{\raise2pt\hbox{$\scriptstyle<$}}{\raise1pt\hbox{$\scriptstyle<$}}{\raise0pt\hbox{$\scriptscriptstyle<$}}{\scriptscriptstyle<}}}
\newcommand{\smallleq}{\mathrel{\mathchoice{\raise2pt\hbox{$\scriptstyle\leq$}}{\raise1pt\hbox{$\scriptstyle\leq$}}{\raise1pt\hbox{$\scriptscriptstyle\leq$}}{\scriptscriptstyle\leq}}}

\newcommand{\boolval}[1]{\mathopen{\lbrack\!\lbrack}\,#1\,\mathclose{\rbrack\!\rbrack}}
\def\[#1]{\boolval{#1}}
\newbox\gnBoxA
\newdimen\gnCornerHgt
\setbox\gnBoxA=\hbox{\tiny$\ulcorner$}
\global\gnCornerHgt=\ht\gnBoxA
\newdimen\gnArgHgt
\def\gcode #1{%
\setbox\gnBoxA=\hbox{$#1$}%
\gnArgHgt=\ht\gnBoxA%
\ifnum     \gnArgHgt<\gnCornerHgt \gnArgHgt=0pt%
\else \advance \gnArgHgt by -\gnCornerHgt%
\fi \raise\gnArgHgt\hbox{\tiny$\ulcorner$} \box\gnBoxA %
\raise\gnArgHgt\hbox{\tiny$\urcorner$}}
\newcommand{\UnderTilde}[1]{{\setbox1=\hbox{$#1$}\baselineskip=0pt\vtop{\hbox{$#1$}\hbox to\wd1{\hfil$\sim$\hfil}}}{}}
\newcommand{\Undertilde}[1]{{\setbox1=\hbox{$#1$}\baselineskip=0pt\vtop{\hbox{$#1$}\hbox to\wd1{\hfil$\scriptstyle\sim$\hfil}}}{}}
\newcommand{\undertilde}[1]{{\setbox1=\hbox{$#1$}\baselineskip=0pt\vtop{\hbox{$#1$}\hbox to\wd1{\hfil$\scriptscriptstyle\sim$\hfil}}}{}}
\newcommand{\UnderdTilde}[1]{{\setbox1=\hbox{$#1$}\baselineskip=0pt\vtop{\hbox{$#1$}\hbox to\wd1{\hfil$\approx$\hfil}}}{}}
\newcommand{\Underdtilde}[1]{{\setbox1=\hbox{$#1$}\baselineskip=0pt\vtop{\hbox{$#1$}\hbox to\wd1{\hfil\scriptsize$\approx$\hfil}}}{}}

\newcommand{\Iff}{\mathrel{\Longleftrightarrow}}

\def\<#1>{\langle#1\rangle}

\newcommand{\TC}{\mathop{{\rm TC}}}

\newcommand{\ZFC}{{\rm ZFC}}
\newcommand{\ZF}{{\rm ZF}}

\newcommand{\KM}{{\rm KM}}
\newcommand{\GB}{{\rm GB}}

\newcommand{\GCH}{{\rm GCH}}

\newcommand{\DC}{{\rm DC}}

\newcommand{\cell}[1]{\boxit{\hbox to 17pt{\strut\hfil$#1$\hfil}}}
\newcommand{\head}[2]{\lower2pt\vbox{\hbox{\strut\footnotesize\it\hskip3pt#2}\boxit{\cell#1}}}
\newcommand{\boxit}[1]{\setbox4=\hbox{\kern2pt#1\kern2pt}\hbox{\vrule\vbox{\hrule\kern2pt\box4\kern2pt\hrule}\vrule}}
\newcommand{\Col}[3]{\hbox{\vbox{\baselineskip=0pt\parskip=0pt\cell#1\cell#2\cell#3}}}
\newcommand{\tapenames}{\raise 5pt\vbox to .7in{\hbox to .8in{\it\hfill input: \strut}\vfill\hbox to
.8in{\it\hfill scratch: \strut}\vfill\hbox to .8in{\it\hfill output: \strut}}}
\newcommand{\Head}[4]{\lower2pt\vbox{\hbox to25pt{\strut\footnotesize\it\hfill#4\hfill}\boxit{\Col#1#2#3}}}
\newcommand{\Dots}{\raise 5pt\vbox to .7in{\hbox{\ $\cdots$\strut}\vfill\hbox{\ $\cdots$\strut}\vfill\hbox{\
$\cdots$\strut}}}
\usepackage[backend=bibtex,style=alphabetic,maxbibnames=15,maxcitenames=6,dateabbrev=false]{biblatex}
\renewcommand{\UrlFont}{\sffamily}
\renewbibmacro{in:}{\ifentrytype{article}{}{\printtext{\bibstring{in}\intitlepunct}}} 
\DeclareFieldFormat{url}{{\UrlFont\url{#1}}} 
\DeclareFieldFormat{urldate}{
  (version \thefield{urlday}\addspace%
  \mkbibmonth{\thefield{urlmonth}}\addspace%
  \thefield{urlyear}\isdot)}
\DeclareFieldFormat{eprint:arxiv}{
  \ifhyperref
    {\href{http://arxiv.org/abs/#1}{%
        arXiv\addcolon\nolinkurl{#1}}\iffieldundef{eprintclass}{}{\UrlFont{\mkbibbrackets{\thefield{eprintclass}}}}}
    {arXiv\addcolon\nolinkurl{#1}\iffieldundef{eprintclass}{}{\UrlFont{\mkbibbrackets{\thefield{eprintclass}}}}}}
\bibliography{HamkinsBiblio,MathBiblio,WebPosts,local}

\begin{document}

\begin{abstract}
We investigate how set-theoretic forcing can be seen as a computational process on the models of set theory. Given an oracle for information about a model of set theory $\<M,\in^M>$, we explain senses in which one may compute $M$-generic filters $G\of\P\in M$ and the corresponding forcing extensions $M[G]$. Specifically, from the atomic diagram of $M$ one may compute a generic filter $G$, from the $\Delta_0$-diagram of $M$ one may compute a presentation of the corresponding forcing extension $M[G]$ and its $\Delta_0$-diagram, and from the elementary diagram of $M$ one may compute the elementary diagram of $M[G]$. We also examine the information necessary to make the process functorial, and conclude that in the general case, no such computational process will be functorial---for any candidate process there will be different isomorphic presentations of a model of set theory $M$ that lead to non-isomorphic forcing extensions $M[G]$. Indeed, there is no Borel function providing generic filters that is functorial in this sense.
\end{abstract}

\maketitle

\section{Introduction}\label{Section.Introduction}

The method of forcing, introduced by Paul Cohen to show the consistency of the failure of the continuum hypothesis,
has become ubiquitous within set theory.
In this paper we analyze this method from the perspective of computable structure theory. To what extent is
forcing an effective process?

\begin{mainquestion*}
Given an oracle for a countable model of set theory $M$, to what extent can we compute its various forcing extensions $M[G]$?
\end{mainquestion*}

We answer this, considering multiple levels of information which might be given by an oracle, whether for the atomic diagram of the model or the $\Delta_0$ diagram or the full elementary diagram, and so forth. We shall generally consider models of set theory presented as structures on the natural numbers equipped with a binary relation for membership, augmented at times with further oracles for the various diagrams of the model.

\begin{maintheorem}[Forcing as a computational process]\label{MainTheorem}\ 
  \begin{enumerate}
    \item Given an oracle for the atomic diagram of a model of \ZF\ set theory $\<M,{\in^M}>$ and a notion of forcing $\P\in M$, there exists a computable procedure (using several further parameters) to decide membership in an $M$-generic filter $G\of^M\P$. 
    \item Given an oracle for the $\Delta_0$-diagram\footnote{We mean $\Delta_0$ in the sense of the \Levy{} hierarchy, as discussed after the theorem and in section~\ref{Section.The-Levy-Diagram}.} of $M$ and a forcing notion $\P\in M$, there exists a uniform computable procedure to decide such a generic filter $G$, the atomic diagram of a presentation of the corresponding forcing extension $M[G]$, and moreover the $\Delta_0$-diagram of $M[G]$.
    \item Given an oracle for the full elementary diagram of $M$ and a forcing notion $\P\in M$, there exists a uniform computable procedure to decide the elementary diagram of such a forcing extension $M[G]$, and this holds level-by-level for the $\Sigma_n$-diagrams.
  \end{enumerate}
\end{maintheorem}

\noindent These statements are proven in Section \ref{Section.Computing-forcing-extensions} as Theorems \ref{Theorem.Computing-generic-filter}, \ref{Theorem.Computing-atomic-diagram-of-extension}, and \ref{Theorem.Computing-diagram-of-extension} below. In Section~\ref{Section.Generic-multiverse} we will extend these results to look at the generic multiverse of a countable model of set theory, i.e.\ those models obtained from the original model by taking both forcing extensions and grounds, where taking grounds is the process inverse to building forcing extensions.  In Section~\ref{Section.Class-forcing} we also consider versions of these results for class forcing instead of set forcing.

In order to avoid a possible confusion, let us emphasize that in this article we use the term $\Delta_0$-diagram to refer to the set of formulae true in the model that are $\Delta_0$ in the \Levy{} hierarchy, rather than the arithmetical hierarchy; and similarly with higher levels of the hierarchy. When we write, say, $\Sigma_n$ without further elaboration, we will always intend $\Sigma_n$ in the \Levy{} hierarchy. In this hierarchy, which is standard in set theory, the $\Delta_0$-formulae are those whose quantifiers are bounded by sets, that is, of the form $\exists x \in y$ or $\forall x \in y$. This differs from $\Delta_0$ in the arithmetical hierarchy, whose formulae have quantifiers bounded by the order relation on $\omega$, that is, of the form $\exists x < y$ or $\forall x < y$. In Section~\ref{Section.Atomic-diagram-knows-little} we will show that very little can be computed from the atomic diagram of a model of set theory. From the atomic diagram we are not able even to compute relations as simple as $x \subseteq y$. We view this as evidence that for the computable structure theory of set theory the \Levy{} $\Delta_0$-diagram is an appropriate choice of the basic information to be used. The usual signature of set theory---just equality and the membership relation---is too spartan to say much of use. In Section~\ref{Section.The-Levy-Diagram} we introduce an expansion of the signature for set theory, which captures the strength of the \Levy{} $\Delta_0$-diagram. We show that the $\Sigma_n$-formulae in the \Levy{} hierarchy are precisely those that are arithmetically $\Sigma_n$ with that expanded signature. That is, when using the fuller signature the distinction between the two hierarchies disappears.

We end the paper, in Sections~\ref{sec:functoriality} and \ref{Section.Non-functoriality}, by investigating how much
information is required to make the process of computing $M[G]$ from $M$ functorial.  Without significant detail
about the dense subsets of $\P$, it will not be so.  Recall that a \emph{presentation} of a countable structure $M$
is simply a structure isomorphic to $M$ whose domain is $\omega$.  The following theorem emphasizes the importance of not conflating the isomorphism type of $M$ with a specific presentation of $M$.

\begin{maintheorem}[Nonfunctoriality of forcing]\label{MainTheorem.Nonfunctoriality}
  There is no computable procedure and indeed no Borel procedure which performs the tasks of Main Theorem \ref{MainTheorem} in a uniform way so that distinct presentations of the model $M$ will result in isomorphic presentations of the forcing extension $M[G]$.
\end{maintheorem}

We also show that forcing can be made a functorial process by adding extra information to the signature of the model. Another way to achieve functoriality is to restrict to a special class of models, namely the pointwise definable models.

\section{The atomic diagram of a model of set theory knows very little}\label{Section.Atomic-diagram-knows-little}

In this section, we show that very little about a model of set theory can be computed from its atomic diagram. In particular, many basic set-theoretic relations are not decidable from the atomic diagram. As a warm-up, let us first see that the atomic diagram does not suffice to identify even a single fixed element.

\begin{proposition}\label{Proposition.Atomic-diagram-knows-little}
 For any countable model of set theory $\<M,\in^M>$ and any element $b\in M$, no algorithm will pick out the number representing $b$ uniformly given an oracle for the atomic diagram of a copy of $M$. 
\end{proposition}

For example, given the atomic diagram of a copy of $M$, one cannot reliably find the empty set, nor the ordinal $\omega$, nor the set $\R$ of reals.

\begin{proof}
Fix $b\in M$ and fix an oracle for the atomic diagram of a copy of $M$. Suppose we are faced with an algorithm purported to identify $b$. Run the algorithm until it has produced a number that it claims is representing $b$. This algorithm inspected only finitely much of the atomic diagram of $M$. The number $b$ it produced has at most finitely many elements in that portion of the atomic diagram. But $M$ has many sets that extend that pattern of membership, and so we may find an alternative copy $M'$ of $M$ whose atomic diagram agrees with the original one on the part that was used by the computation, but disagrees afterwards in such a way that the number for $b$ now represents a different set in $M'$. So the algorithm will get the wrong answer on $M'$.
\end{proof}

This idea can be extended to characterize which relations on $M$ are computable from the atomic diagram. In particular, any such relation must contain both finite and infinite sets.

\begin{theorem}\label{Theorem.Sets-known-by-atomic-diagram}
Let $\<M,\in^M>\satisfies\ZF$ be a countable model of set theory and let $X \subseteq M^n$ for some $1 \le n < \omega$. The following are equivalent.
\begin{enumerate}
\item $X$ is uniformly 
relatively intrinsically computably enumerable in the atomic diagram of $M$. That is, there is a single computably enumerable operator which given the atomic diagram of a presentation of $M$  will output the copy of $X$ for that presentation.
\item Membership of each single $\avec$ in $X$ is witnessed by a finite pattern of $\in$ in the transitive closures of $\{a_0\}, \ldots, \{a_{n-1}\}$, and the list of finite patterns witnessing membership is enumeration-reducible
to the purely existential diagram of $M$.
\item There is an existential $\mathcal{L}_{\omega_1,\omega}$-formula $\varphi(x)$ which defines $X$, and the set of (finite) disjuncts in this formula is enumeration-reducible to the purely existential diagram of $M$.
\end{enumerate}
\end{theorem}

\begin{proof}
$(1 \Leftrightarrow 3)$ is a standard fact in computable structure theory, first established in \cite{AKMS},
relativized here since $M$ does not have a computable presentation.

$(2 \Rightarrow 1)$ For notational simplicity we will present the argument for the $n = 1$ case. Suppose membership of $a$ in $X \subseteq M$ is witnessed by a finite pattern of $\in$ in the transitive closure of $\{a\}$. That is, $a$ is in $X$ if and only if one of a certain list of finite graphs can be found in the pointed graph $(\TC(\{a\}),a,\in^M)$. By hypothesis
this list is $e$-reducible to the arithmetical $\Sigma_1$-diagram of $M$, which we can enumerate,
since we have the atomic diagram of $M$ as an oracle.  Therefore, we can list out these finite graphs,
one by one.  Meanwhile, we enumerate $\in^M$ and $M$, and continually check whether the most recent pair in $\in^M$ has completed a copy of a graph from this list. If such happens, then we search through the $a \in M$ we have already enumerated and check whether $a$ has one of these graphs appearing below its transitive closure. We output all the $a$ for which this happens. We also check whether the most recent $a \in M$ has a copy of one of the graphs in the portion of its transitive closure so far enumerated. If that happens, we output $a$. This process will find every $a \in M$ which has a copy of this graph below its transitive closure, so it will exactly enumerate $X$.

$(1 \Rightarrow 2)$ Again, we will we present the proof for the $n = 1$ case for notational simplicity. Assume that $X$ is uniformly relatively intrinsically computable enumerable from the atomic diagram of $M$. The conclusion that an $a$ lies in $X$ must be based on a finite portion of the graph $\in^M$ where we are only allowed to know the index of $a$ itself. We can separate this information into two pieces, that inside the transitive closure of $\{a\}$ and that outside. We claim that only the first information can be relevant, which then implies $(2)$. To see this, observe that every possible finite pattern of finite graphs (with $a$ as a constant) occurs outside of the transitive closure of $\{a\}$ in every model of set theory. So the information outside is true for every $a \in M$ and thus cannot have any effect on the decision of whether to output $a$ in the enumeration.
\end{proof}

As a consequence of this theorem, the atomic diagram of a model of set theory does not know any of the predicates
and relations named in Lemma \ref{lemma:Delta1} below, which are each easily checked to fail item $(2)$ of Theorem~\ref{Theorem.Sets-known-by-atomic-diagram}.
In short, the atomic diagram of a model of set theory knows little about the model.
Before addressing forcing, therefore, we now present a natural expansion of the signature,
under which the atomic diagram will present the information that set theorists normally consider to be basic.

\section{The \Levy{} diagram}\label{Section.The-Levy-Diagram}

The signature ordinarily used in set theory is simple: it consists of the binary relation $\in$ along with equality. This
signature suffices all by itself to express all of mathematics (possibly using assumptions beyond \ZFC). But from the perspective of computable structure theory we have observed that we cannot say much using just the atomic diagram in this signature. As we see it, the lesson here is that this atomic diagram is too weak to take as the basic information for the computable structure theory of models of set theory. Instead, we will work with the $\Delta_0$-diagram, in the sense of the \emph{\Levy\ hierarchy}, in which set-bounded quantification is not counted when determining quantifier complexity of a formula. 

This terminology and usage conflicts, unfortunately, with the standard convention in computability theory, where ``bounded quantification'' refers to bounded in the sense of arithmetic---i.e.\ of the form $\exists x<y$ or $\forall x<y$---rather than bounded in the sense of a relation of the structure. In short, $\Delta_0$ in the \Levy{} hierarchy is not the same as $\Delta_0$ in the arithmetical hierarchy for the presentation of the structure, as $\forall x\in a\ \varphi(x)$ may express infinitely many independent atomic statements.

In this section we describe how we can make the \Levy{} hierarchy line up with the arithmetical hierarchy, by
expanding the signature to capture the content of the (\Levy) $\Delta_0$-diagram. Let us call this expanded signature the \emph{\Levy{} signature}. This provides an alternate way to think of the $\Delta_0$-diagram. One may instead work with the atomic diagram in this expanded signature, with the correspondence between the two being effective. And this correspondence continues upward, with $\Sigma_n$ in the \Levy{} hierarchy corresponding to $\Sigma_n$ in the arithmetical hierarchy with this expanded signature. In the sequel, we will speak of the $\Delta_0$-diagram or $\Sigma_n$-diagram, referring to the \Levy{} hierarchy, but the reader who prefers the other approach may freely translate over the statements.\footnote{As further evidence for the naturalness of this expansion of the language, we note that it has arisen in other contexts. Venturi and Viale \cite{Venture-Viale:submitted} studied model companions for set theory. For their work they also found it helpful to expand the signature from the spartan $\{\in\}$ to include symbols for each $\Delta_0$ relation.}

\begin{definition}\label{Definition.Levy.signature}
The \emph{\Levy{} signature} contains the binary relation symbol $\in$ and also, for each $\Delta_0$-formula $\varphi(x_1,\ldots,x_n)$, an $n$-ary predicate $R_{\varphi}$.  Structures in the \Levy{} signature are assumed to satisfy the stipulative definition axiom schema (over all $\Delta_0$-formulae): 
$$ \forall \vec x\ [R_{\varphi}(\vec x) \ \Iff \ \varphi(\vec x)].$$
\end{definition}
For example, since inclusion is defined by a $\Delta_0$-formula using only $\in$, we see that 
$$ M\models b\subseteq c \ \Iff \ M\models \forall y\in b\ y\in c
\ \Iff \ M\models R_{(\forall y\in x_1)y\in x_2}(b,c).$$
Since we can enumerate the $\Delta_0$-formulae effectively, the language of the \Levy{} signature is a computable language, and it includes equality (defined by $x\subseteq y \land y\subseteq x$).

Our choice of signature allows us to imitate the usual set-theoretic convention under which $\Delta_0$ is the lowest possible complexity of a formula.
Additionally, sets defined by $\Sigma_n$-formulae in the \Levy{} hierarchy will all be arithmetically $\Sigma_n$
in models of \ZF\ in the \Levy{} signature, by the following standard lemma.
\begin{lemma}
\label{lemma.Levy.signature}
For every formula $\varphi(t_1,\ldots,t_n,x,y,z)$ in the \Levy{} hierarchy,
$$ \ZF \proves \forall\vec t\ 
\bigl[\forall x\in z\ \exists y\ \varphi(\vec t,x,y,z) \Longleftrightarrow \exists Y\ \forall x\in z\ \exists y \in Y\ \varphi(\vec t,x,y,z)\bigr].$$ 
\end{lemma}

\begin{proof}
If $M \models \forall x\in z\ \exists y\ \varphi(\vec a,x,y,z)$ for a model $\<M,\in^M>$ of $\ZF$ with parameters $\vec a$, then by the Replacement axiom schema, there is some set $Y$ in $M$ such that $M \models \forall x\in z\ \exists y\in Y\ \varphi(\vec a,x,y,z)$.  The reverse direction is trivial.
\end{proof}

The lemma, applied repeatedly, allows us to turn every formula from the signature with just membership and equality into an equivalent formula which is in a prenex form, with all unbounded quantifiers preceding a $\Delta_0$ matrix, and which has the same complexity (in the \Levy{} hierarchy) as the original formula. Therefore, every formula of \Levy{} complexity $\Sigma_n$ may be expressed in the \Levy{} signature by a formula of arithmetic complexity $\Sigma_n$, and there is an effective procedure for producing the latter from the former. 

To conclude this section, let us remark that many basic predicates of set theory are $\Delta_0$, and that the basic predicates for forcing are $\Delta_1$ and hence computable from the $\Delta_0$-diagram.

\begin{lemma}
\label{lemma:Delta1}
For models $M$ of \ZFC, the following predicates are all $\Delta_0$.
\begin{multicols}{2}
\begin{itemize}[label={}]
\item $x=\emptyset$.
\item $x \subseteq y$.
\item $x = \{y,z\}$.
\item $x = \bigcup y$.
\item $x = \{ z \in y : \varphi(z) \}$, for $\varphi\in\Delta_0$.
\item $x$ is a (Kuratowski) ordered pair.
\item $x$ is a set of ordered pairs.
\item $x$ is a function.
\item $x$ is transitive.
\item $x$ is an ordinal.
\item $x$ is inductive.
\item $x=\omega$.
\end{itemize}
\end{multicols}
\end{lemma}

We omit any proof of these standard facts. More pertinently for this article, let us briefly remark that the relations ``$x$ is a $\P$-name'', $p \forces \sigma \in \tau$, and $p \forces \sigma = \tau$ are all $\Delta_1$ and hence computable from the $\Delta_0$-diagram. This will be important in Section~\ref{Section.Computing-forcing-extensions} to see that the $\Delta_0$-diagram suffices to compute a presentation of the forcing extension.

\section{Computing forcing extensions}\label{Section.Computing-forcing-extensions}

For the sake of the reader who may not be an expert in set theory, we interleave our proof of the first main theorem with an exposition of forcing.\footnote{For full details we refer the reader to \cite{Jech:SetTheory3rdEdition} or \cite{Kunen1980:SetTheory}. See also \cite{Chow2007:ABeginnersGuideToForcing} for a conceptual overview.}
This exposition will follow the three parts of the main theorem. First we discuss generic filters and we show that a generic filter may be effectively built from just the atomic diagram. Next we discuss how to build the forcing extension $M[G]$ from the ground model $M$ and the generic $G$. Then we show that there is an effective procedure to construct (the atomic diagram of) the forcing extension from the $\Delta_0$-diagram of the ground model. Finally, we discuss how truth in the forcing extension is determined by the forcing relations in the ground model. In particular, we show that the elementary diagram of the forcing extension is effectively computable from the elementary diagram of the ground model. 

In the ground universe $M$, a partially ordered set $\P$ gives partial information about a certain ideal object, a generic filter, which can be thought to exist outside the given universe.  A condition $p$ is stronger than another condition $q$, written $p \le q$, if it gives more information about this outside object. It is convenient to assume that $\P$ has a maximum element $1$ which gives no information.
Two conditions are compatible, written $p \parallel q$, if there is a condition stronger than both of them. Otherwise, they are incompatible, written $p \perp q$. Formally, the outside object is an $M$-generic filter $G \subseteq \P$. That is, it  is downward directed and upward closed, and it meets every dense $D \subseteq \P$ in $M$. Here, $D \subseteq \P$ is dense if every condition in $\P$ can be strengthened to a condition in $D$. It is a straightforward exercise that if $\P$ is 
splitting---any condition $p$ extends to two incompatible conditions---then $M$ itself contains no $M$-generic filter. 

However, if $M$ is countable, we can always construct an $M$-generic filter $G \subseteq \P$ externally to $M$,
and the construction works uniformly for all $\P$ and uniformly in $\Delta_0(M)$ for any presentation of $M$.  We simply fix the first-encountered $p_0\in\P$
as our starting point (using the external order $<$ on the domain $\omega$ of $M$ to find it)
and then ask of each element $s\in\omega$ in turn whether the current $p_s$ can be extended to an
element of $\P\cap s$. This is a $\Delta_0$ question, and we either search for and find a $p_{s+1}\leq p_s$
lying in $s$ if the $\Delta_0(M)$ indicates that one exists, or else set $p_{s+1}=p_s$.
Finally, set $G = \{ q \in \P : q \parallel p_n$ for some $n\}$.  Of course, if $s$ was a dense
subset of $\P$, then we ensured that $G$ does meet $s$, so this $G$ is $M$-generic.

In fact, there exists a generic filter that is computable merely from the atomic diagram, as we now  show.

\begin{theorem}\label{Theorem.Computing-generic-filter}
Given an oracle for the atomic diagram of a model of set theory $\<M,\in^M>\satisfies\ZF$ and a notion of forcing $\P\in M$, there exists a computable procedure (using several further parameters) to decide membership in an $M$-generic filter $G\of\P$. 
\end{theorem}

\begin{proof}
Fix an oracle for the atomic diagram of $\<M,\in^M>$. Let $\P\in M$ be any notion of forcing in $M$. More specifically, we have in $M$ an element $\P\in M$ for the underlying set of the partial order, and we also have a set $\leq_\P$ in $M$ for what $M$ thinks is the set of ordered pairs for that relation. We assume that the Kuratowski pairing function  $(p,q)=\set{\set{p},\set{p,q}}$ is used when coding ordered pairs.

Notice that from the atomic diagram we can decide whether a given number $p$ represents an element of $\P$ or not, since we need only ask the oracle whether $p\in^M\P$. Further, we can computably enumerate the pairs $(p,q)$ of forcing conditions with $p\leq_\P q$. For this we must unwrap the Kuratowski pairing function, but this is possible as follows. We search for conditions $p,q$ in $\P$ and for elements $x\in{\leq_\P}$ that represent the pair $(p,q)$. This will happen when $x=\set{y,z}$, where $y=\set{p}$ and $z=\set{p,q}$. So we can search for the elements $y$ and $z$ which have $y\in^M x$ and $z\in^M x$ and $p\in^M y$ and $p,q\in^M z$. When this happens, we can be confident that $p\leq_\P q$. (This kind of computational inspection of the Kuratowski ordered pair also arose in \cite{GodziszewskiHamkins2017:Computable-quotient-presentations-of-models-of-arithmetic-and-set-theory}.)

But actually, we can fully decide the relation $\leq_\P$, not just enumerate it. The reason is that $M$, being a model of set theory, has what it thinks is the set of pairs $(p,q)\in\P^2$ with $p\not\leq_\P q$. And so for any pair $p,q$, we can search for it to be enumerated by $\leq_\P$ or by the analogous procedure applied with $\not\leq_\P$, and thereby decide whether $p\leq_\P q$ or not.

Next, let $\mathcal{D}\in M$ be the set that $M$ thinks is the set of all dense subsets of $\P$. From this data, we can enumerate the elements $D_0,D_1,D_2,\ldots$ of $\mathcal{D}$, listing all the elements of $M$ that $M$ thinks are dense subsets of $\P$. We simply run through all the natural numbers $d$, and ask the oracle whether $d\in^M \mathcal{D}$, and if so, we put it on the list.

Let us use this enumeration to compute a descending sequence of forcing conditions
 $$p_0\quad\geq_\P\quad p_1\quad\geq_\P\quad p_2\quad\geq_\P\quad\cdots$$
with $p_n\in D_n$. To begin, we let $p_0$ be the first-encountered element of $D_0$, which exists because $D_0$ is dense. Next, given $p_n$, we search through the natural numbers for the first-encountered condition $p_{n+1}\leq_\P p_n$ such that $p_{n+1}\in D_{n+1}$. There is always such a condition because $D_{n+1}$ is dense.

Finally, having constructed the descending sequence, we define $G$ as the set of conditions $q$ for which $p_n\leq q$ for some $n$. This is an $M$-generic filter, because it is a filter and it meets every dense subset of $\P$ in $M$. The elements of $G$ can be enumerated by the processes above, because whenever we find a condition $q$ for which $p_n\leq_\P q$ for some $n$, then we can enumerate $q$ into the filter.

But actually, we claim that $G$ is fully decidable from the oracle, not just enumerable. To see this, let $\perp_\P$ be the set in $M$ that $M$ thinks is the set of pairs $(p,q)$ of incompatible conditions $p\perp q$. By genericity, it follows that if $q\notin G$, it must be that $q\perp_\P p_n$ for some $n$, since it is dense to either get below $q$ or become incompatible with it. The algorithms above allow us to enumerate the pairs of incompatible conditions, and so for any condition $q$, we search for a $p_n$ for which either $p_n\leq_\P q$ or $p_n\perp_\P q$, and in this way we can tell whether $q\in G$ or $q\notin G$, as desired.
\end{proof}

The algorithm above is non-uniform in several senses. First, the algorithm makes use not only of the indices
of the forcing notion (namely, the set $\P$, the set $\leq_{\P}$ and the complement of $\leq_{\P}$ in $\P\times\P$),
but also the index of the set of dense subsets of $\P$, and the index $\perp_\P$ for the set of incomparable elements. The methods of Theorem \ref{Theorem.Sets-known-by-atomic-diagram} show that it is not possible in general to compute these just from the atomic diagram of $\<M,\in^M>$ and $\P$. Second, a more serious kind of non-uniformity arises from the fact that, even if we are given this additional data and even if we are given the full elementary diagram of the model $M$, the particular generic filter that we end up with will depend on the order in which $M$ is represented in this presentation. The filter $G$ is determined in part by the order in which the dense sets $D_n$ appear in the presentation of $M$. If one rearranges the dense sets, then at a certain stage, one might be led to place a different, incompatible condition on the sequence, and this will give rise to a different filter.
In Sections \ref{sec:functoriality} and \ref{Section.Non-functoriality}, we shall examine uniformity in more detail and prove there is no uniform computable procedure nor even Borel function that always
produces the same generic filter $G$ from all isomorphic presentations of the same model $M$.

We also want to remark that it mattered which pairing function we used. Given sets $x,y,p$ where we know $p$ is a Kuratowski ordered pair it requires looking at only finitely many objects to check whether $p = (x,y)$. On the other hand, there are $\Delta_0$ definitions for ordered pairs which we cannot computably unravel from just the atomic diagram. Consider the Morse pairing function 
$$(x,y)^\star = (\set{0} \times x) \cup (\set{1} \times y),$$
where the product here is defined via the Kuratowski ordered pair.\footnote{This pairing function has the nice property that if $x,y \subseteq V_\alpha$ where $\alpha$ is a limit ordinal, then $(x,y)^\star \subseteq V_\alpha$.}
This definition is $\Delta_0$ in the \Levy{} hierarchy. But to recognize whether $p = (x,y)^\star$ requires looking at infinitely many elements of the model when at least one of $x$ and $y$ is infinite, making it not effective (from just the atomic diagram). It follows from Theorem~\ref{Theorem.Sets-known-by-atomic-diagram} that the relation $p = (x,y)^\star$ is not decidable from the atomic diagram, even if we know $p$ is a Morse pair.
\smallskip

We next turn to the construction of the full forcing extension. 
Given a generic filter $G$, which we have already seen how to construct, we need to determine the rest of the sets that will constitute the forcing extension of $M$. New sets are given by names in $M$, which are then interpreted by the generic.
These $\P$-names are recursively defined as sets whose elements are pairs $(\sigma,p)$ with $\sigma$ a $\P$-name and $p \in \P$. This amounts to a recursive definition on rank, with each $\P$-name having an ordinal rank.

\begin{lemma}
\label{lemma:Pnames}
The property of being a $\P$-name is $\Delta_1$, hence decidable uniformly in the $\Delta_0$-diagram of $M$.
\end{lemma}

Naively, one might attempt to prove that the class of $\P$-names is $\Sigma_1$ by, given an element $x$, querying the $\Delta_0$-diagram to check that all elements of $x$ are pairs whose first coordinate is a $\P$-name, which is checked by querying the $\Delta_0$-diagram, and recursively continue this process until it halts. (And similarly to prove the class of $\P$-names is $\Pi_1$.) The trouble with this approach is that this recursive procedure is transfinite; even in the case it is internally finite, $M$ might be $\omega$-nonstandard and thus have internally finite sets which are externally seen to be infinite. Nevertheless, the key of this idea is correct, and it can be made into a proper proof. The point is that instead of externally carrying out the recursive procedure, we instead look inside $M$ for a certificate that the recursive procedure was carried out.

\begin{proof}[Proof of Lemma~\ref{lemma:Pnames}]
To see that being a $\P$-name is $\Sigma_1$, given a set $x$ we check whether $x$ is a $\P$-name by searching for a tree witnessing that $x$ satisfies the recursive definition of being a $\P$-name. Such a tree has $x$ as its root, and the immediate children of the root are the first coordinates of elements of $x$. These nodes then have their own children by the same process, and this continues downward through the whole tree. This tree is necessarily well-founded (in the sense of $M$), because $M$ thinks its membership relation is well-founded. And this tree witnesses that $x$ is a $\P$-name if the second coordinate of each node is in $\P$ and no node has an element not represented among its children. It is clear that given a tree $T$ it is $\Delta_0$ to check whether it satisfies this property of witnessing that $x$ is a $\P$-name, since this requires only quantifying over nodes in the tree and over $\P$. And thus the class of $\P$-names is $\Sigma_1$.

To see that it is also $\Pi_1$ we carry out a similar procedure, except now we look for a tree witnessing that the recursive definition of being a $\P$-name fails for $x$. Being such a tree is again $\Delta_0$, because failure is witnessed by looking at and below some node. This then gives that being a $\P$-name is a $\Delta_1$ property. 
\end{proof}

We would like to remark that this argument generalizes; any property defined by a set-theoretic recursion of $\Delta_0$ properties will be $\Delta_1$.
\smallskip

A standard approach to forcing is to define the interpretation of a name $\sigma$ by $G$, denoted $\sigma_G$, to be the set of $\tau_G$ for $(\tau,p) \in \sigma$ for some $p \in G$. But this recursive interpretation can only be carried out if $M$ is well-founded, limiting which models it can be applied to. Fortunately, there is an alternative construction we can use, which applies to all models and which matches the recursive interpretation construction in case $M$ is well-founded. 

Every set theorist knows that there are a pair of definable relations,
\begin{align*}
p &\forces_\P \sigma = \tau \\
p &\forces_\P \sigma \in \tau,
\end{align*}
which determine membership and equality in the forcing extension. Namely, $\sigma_G \in \tau_G$ if and only if there is $p \in G$ so that $p \forces_\P \sigma \in \tau$, while $\sigma_G = \tau_G$ if and only if there is $p \in G$ so that $p \forces_\P \sigma = \tau$.
We can use this property as the definition of the forcing extension. That is, define the following relations on the class of $\P$-names:
  $$\begin{array}{cl}
      \sigma =_G \tau\quad &\Iff \quad \exists p\in G\quad p\forces_\P \sigma=\tau \medskip\\ 
      \sigma \in_G \tau\quad &\Iff \quad \exists p\in G\quad p\forces_\P \sigma\in\tau.\\
  \end{array}
  $$
It is readily checked that the relation $=_G$ is an equivalence relation and indeed a congruence relation with respect to the relation $\in_G$. The forcing extension $M[G]$ can then be presented as the equivalence classes of names $[\sigma]_G$ by the relation $=_G$ with the membership relation induced by $\in_G$. This is essentially the Boolean ultrapower manner of constructing the forcing extension, rather than the value recursion method (see \cite{HamkinsSeabold:BooleanUltrapowers} for an account of how these constructions can differ).

Let us see that the above described process is indeed effective in the $\Delta_0$-diagram.

\begin{theorem}\label{Theorem.Computing-atomic-diagram-of-extension}
There is a uniform computable procedure that, given an oracle for the $\Delta_0$-diagram of a model of set theory $\<M,\in^M> \satisfies \ZF$ and a notion of forcing $\P \in M$, decides membership in an $M$-generic filter $G \of \P$ and then computes the atomic diagram of a presentation of the extension $M[G]$. Indeed, with this oracle it can decide the $\Delta_0$-diagram of this presentation.
\end{theorem}

\begin{proof}
The diagram of the forcing extension $M[G]$ will be in the full forcing language 
 $$\langle M[G],\in^{M[G]},\check M,\sigma\rangle_{\sigma\in M^\P},$$ 
with a predicate $\check M$ for the ground model and constants for all the $\P$-names $\sigma$; we denote the class of $\P$-names in $M$ by $M^\P$. The class $M^\P$ of $\P$-names was seen in Lemma \ref{lemma:Pnames}
to be computable from $\Delta_0(M)$.
The atomic forcing relations
\begin{align*}
p &\forces_\P \sigma = \tau \\
p &\forces_\P \sigma \in \tau,
\end{align*}
have complexity $\Delta_1$ with $\P$ as a parameter, because again these relations are the result of the solution of a recursion of a $\Delta_0$ property (see \cite{GitmanHamkinsHolySchlichtWilliams2020:The-exact-strength-of-the-class-forcing-theorem} for explicit discussion of the forcing-relation-as-solution-to-a-recursion perspective). Similarly the negated atomic forcing relations are also $\Delta_1$ with $\P$ as a parameter. 
So given the $\Delta_0$-diagram and knowing membership in the generic $G$ we can compute $=_G$, and thereby pick out representatives from the equivalence classes and compute a bijection onto $\omega=\dom{M[G]}$ from the collection of these classes: $0$ will
denote the class of the $<$-least $\P$-name, $1$ the class of the least $\P$-name not in the class of $0$, and so on. Similarly, we can also compute $\in_G$, giving us the atomic diagram of $M[G]$ (in the signature with just $\in$).

We delay the argument that we can decide the $\Delta_0$-diagram of $M[G]$ until after a discussion of truth in the forcing extension; see the proof of Theorem~\ref{Theorem.Computing-diagram-of-extension}.
\end{proof}

At last we come to truth in the forcing extension. As a first-order structure, truth in $M[G]$ is given by the usual Tarskian recursive construction. A key fact in forcing, however, is that truth in the forcing extension is closely tied to the ground model. We have already seen the start of this with the atomic forcing relations $p \forces \sigma \in \tau$ and $p \forces \sigma = \tau$. In general, for any formula $\varphi(x_0,\ldots,x_n)$ there is a definable relation $p \forces \varphi(\sigma_0,\ldots,\sigma_n)$ so that $M[G] \models \varphi((\sigma_0)_G,\ldots,(\sigma_n)_G)$ if and only if there is $p \in G$ so that $p \forces \varphi(\sigma_0,\ldots,\sigma_n)$. The map which sends $\varphi$ to the formula $p \forces \varphi$ is effective. And as we explain in the proof, this map does not increase complexity; if $\varphi$ is $\Delta_0$ then $p \forces \varphi$ is $\Delta_1$ and if $\varphi$ is $\Sigma_n$ (respectively $\Pi_n$) for $n \ge 1$ then $p \forces \varphi$ is $\Sigma_n$ (respectively $\Pi_n$).

Using these forcing relations, we can compute the elementary diagram of the extension if we are given the elementary diagram of the ground model. Indeed, this goes level by level.

\begin{theorem}\label{Theorem.Computing-diagram-of-extension}
There is a uniform computable procedure that, given an oracle for the full elementary diagram of a model of set theory $\<M,\in^M>\satisfies\ZF$ and a notion of forcing $\P\in M$, decides the full elementary diagram of the presentation of $M[G]$ built in Theorem \ref{Theorem.Computing-atomic-diagram-of-extension}.  Moreover, this goes level-by-level: given $\P$ and an oracle for the $\Sigma_n$-elementary diagram of $M$, for $n \ge 1$, it can decide the $\Sigma_n$-elementary diagram of this presentation.
\end{theorem}

\begin{proof}
By the process of Theorem~\ref{Theorem.Computing-atomic-diagram-of-extension} we can decide membership in a generic $G$ and compute the atomic diagram of a presentation of $M[G]$. Consider next the forcing relations $p\forces_\P\varphi(\sigma_0,\ldots,\sigma_n)$, where $p\in\P$ and $\varphi$ is an assertion in the language of set theory with $\P$-name parameters $\sigma_i\in M^\P$. For each formula $\varphi$, the corresponding forcing relation (as a relation on $p$ and the names $\sigma_i$) is definable in the ground model $M$. Furthermore, the proof that the forcing relations are definable is uniform, in the sense that from any formula $\varphi$, we can write down the formula defining the corresponding forcing relation.\footnote{We are not claiming that the forcing relations are uniformly definable in $M$, since indeed as $\varphi$ increases in complexity, the complexity of the definition of $p\forces\varphi(\sigma)$ similarly rises. Rather, we only claim here that there is a computational procedure that maps any formula $\varphi$ to the formula $\text{force}_\varphi(p,\sigma)$ defining the forcing relation ``$p\forces\varphi(\sigma)$'' in $M$.}
So it suffices to prove now that from the $\Sigma_n$-diagram of the ground model we can compute the $\Sigma_n$-diagram of the extension. Given the full diagram of the ground model, the same process will compute the full diagram of the extension.

First, we claim that the forcing relation $p \forces \varphi(\sigma_0, \ldots, \sigma_n)$ for any $\Delta_0$-formula $\varphi$ is complexity $\Delta_1$ in the \Levy\ hierarchy. This can be proved by induction on formulae. The only nontrivial case is the bounded-quantifier case $p\forces\exists x\in\tau\ \varphi(x,\sigma)$, which by the forcing recursion is equivalent to saying that there is a dense collection of conditions $q\leq p$ with some $\<\rho,r>\in\tau$ such that $q\leq r$ and $q\forces\varphi(\rho,\sigma)$. The point is that all these quantifiers remain bounded, and $\Delta_1$ is closed under bounded quantification in set theory. This establishes that given the $\Delta_0$-diagram of $M$ we can compute the $\Delta_0$-diagram of $M[G]$, completing the proof of Theorem \ref{Theorem.Computing-atomic-diagram-of-extension}.

We can now prove inductively that the forcing relation $p\forces\varphi(\sigma)$ for $\Sigma_n$-formulae $\varphi$ has complexity $\Sigma_n$, for $n\geq 1$, and similarly the forcing relation on $\Pi_n$-formulae is $\Pi_n$.\footnote{Note that this argument uses the Replacement axiom schema. This is as $p \forces \exists x \varphi(x)$ if and only if there are densely many $q \le p$ so that $q \forces \varphi(\tau)$ for some name $\tau$. It is the Replacement schema that allows us to pull the bounded quantifier over $q$ inside the unbounded quantifiers in the definition for $q \forces \varphi(\tau)$ to obtain a $\Sigma_n$-formula. \label{Footnote.Replacement}} 
We can tell if $M[G]\satisfies\varphi(\sigma_0,\ldots,\sigma_n)$, for $\varphi$ a $\Sigma_n$-formula, by looking for a condition $p\in G$ such that $M$ satisfies the assertion $p\forces_\P\varphi(\sigma_0,\ldots,\sigma_n)$.\footnote{Recall that in this context we are working in the full forcing language with constants for every $\P$-name, so it is sensible to ask whether $M[G]$ satisfies $\varphi(\sigma_0,\ldots, \sigma_n)$ where the $\sigma_i$ here are constant symbols referring to $(\sigma_i)_G$.}
Thus, the $\Sigma_n$-diagram of the forcing extension $M[G]$ can be computed from the $\Sigma_n$-diagram of $M$, for $n\geq 1$, as desired.
\end{proof}

We can also prove a version of this theorem for computable infinitary formulae. By definition the $\Sigma_0^c$-formulae are the Boolean combinations of atomic formulae---which, since we work in the \Levy\ signature, means precisely the $\Delta_0$-formulae---and these are also the $\Pi_0^c$-formulae.  For computable ordinals $\alpha$, the $\Sigma_{\alpha+1}^c$-formulae are those $\Sigma_{\alpha+1}$-formulae in $\mathcal{L}_{\omega_1,\omega}$ with finitely many free variables $\xvec$ that are computable (countable) disjunctions of formulae $\exists \yvec_m \gamma_m(\xvec, \yvec_m)$, with each $\gamma_m$ in $\Pi_{\alpha}^c$. (The length $k_m$ of the tuple $\yvec_m=(y_{m,1},\ldots,y_{m,{k_m}})$ of variables may vary over $m$, but must be computable from $m$.) The $\Pi_{\alpha+1}^c$-formulae are their negations. For computable limit ordinals $\alpha$, the $\Sigma_{\alpha}^c$-formulae are computable disjunctions in which each individual disjunct is $\Pi_{\beta}^c$ for some $\beta < \alpha$; any computable presentation of $\alpha$ may be used in this definition to give the $\beta$'s.

\begin{theorem} \label{Theorem.Computing-infinitary-diagram}
Fix a computable ordinal $\alpha$.  Then there is a computable function $f:\omega^2\to\omega$
such that, for every model of set theory $\<M,\in^M>\satisfies\ZF$
and every notion of forcing $p=\P\in M$, the function
$f(p,~\cdot~):\omega\to\omega$ is an $m$-reduction from the $\Sigma_\alpha^c$-diagram
of the structure $M[G]$ produced by the procedure in
Theorem \ref{Theorem.Computing-atomic-diagram-of-extension}
to the $\Sigma_\alpha^c$-diagram of $M$ itself.
(Here $p\in\omega=\dom{M}$ is the domain element $\P$.)
\end{theorem}

\begin{proof}
Knowing that the characteristic function of $\Delta_0(M[G])$ is given as $\Psi^{\Delta_0(M)}$
for some Turing functional $\Psi$, we explain the $m$-reduction between the $\Sigma_1^c$-diagrams
claimed in the theorem.  A $\Sigma_1^c$-formula
$$\bigvee_{m\in\omega} \exists\yvec_m~\gamma_m([\sigma_1]_G,\ldots,[\sigma_k]_G,\yvec_m), $$
using a computable sequence $\langle\gamma_m\rangle_{m\in\omega}$ of $\Delta_0$-formulae about
a finite tuple from $M[G]$, holds in $M[G]$ just if there exist $m,s\in\omega$, elements
$[\tau_1]_G,\ldots,[\tau_{k_m}]_G$ in $M[G]$, and a finite initial segment $\rho\subseteq\Delta_0(M)$
such that $\Psi^{\rho}$ converges within $s$ steps on the G\"odel number of the formula
$\gamma_m([\sigma_1]_G,\ldots,[\sigma_k]_G,[\tau_1]_G,\ldots,[\tau_{k_m}]_G)$
and outputs $1$, meaning that this $\Delta_0$-formula holds in $M[G]$.  This constitutes
a $\Sigma_1^c$ statement about $M$ itself, quantifying over the $\rho\subseteq\Delta_0(M)$
which cause the program $\Psi$ to halt with value $1$ (as well as over $m$, $s$, and the $\P$-names).
Since we can compute an index for this $\Sigma_1^c$ statement about $M$ from the original
formula about $M[G]$, we have an $m$-reduction from $\Sigma_1^c(M[G])$ to $\Sigma_1^c(M)$,
as claimed.  This same function is an $m$-reduction from $\Pi_1^c(M[G])$ to $\Pi_1^c(M)$,
and analogous arguments hold with any larger computable ordinal $\alpha+1$ in place of $1$,
and also for limit ordinals.
\end{proof}

To close off this section, we remark that the same analysis applies to computing symmetric extensions, used to produce models where the axiom of choice fails. These extensions are generally obtained by restricting the $\P$-names to a certain class of symmetric names. Namely, we fix a group $G$ of automorphisms of $\P$ and a normal filter $\mathcal{F}$ on the subgroups of $G$ and then define a $\P$-name to be \emph{symmetric} if the subgroup of $G$ consisting of permutations that fix the name is in $\mathcal{F}$. The class of \emph{hereditarily} symmetric names are thus constructed by a transfinite recursion, akin to the recursion defining the names. This is again a recursion where each stage is $\Delta_0$, and so the class of hereditarily symmetric names is $\Delta_1$. Carrying out much the same argument as for full forcing extensions, one can obtain versions of Main Theorem~\ref{MainTheorem} \textit{(2)} and \textit{(3)} for symmetric extensions.

\section{The generic multiverse}\label{Section.Generic-multiverse}

In the previous section we investigated the computable structure theory of how a model $M$ of set theory relates to a single forcing extension $M[G]$. We turn now to a broader perspective. Given a model $\<M,\in^M>$ of set theory, the \emph{generic multiverse} of $M$ is the smallest collection of models of set theory which is closed under extension by forcing and by grounds, where $W$ is a \emph{ground} of $M$ if $M$ is a forcing extension of $W$ via a partial order in $W$. In this section we would like to investigate the extent to which the generic multiverse can be computed from a countable $M$, extending the analysis in Section \ref{Section.Computing-forcing-extensions}. Let us begin by looking at grounds.

\begin{lemma}\label{Lemma.Computing-grounds}
There is a uniform computable procedure which given an oracle for the $\Pi_2$-elementary diagram of a model of set theory $\<M,\in^M> \models \ZFC$ will compute a list of the $\Delta_0$-diagrams of the grounds of $M$. 
\end{lemma}

\begin{proof}
Because the membership relation of a ground of $M$ is the restriction of the membership relation of $M$ and thus $M$ and its grounds agree on $\Delta_0$ truth, all we need to compute is the domains of the grounds.
The key fact is the ground model enumeration theorem \cite[Theorem 12]{FuchsHamkinsReitz2015:Set-theoreticGeology}, which asserts that the grounds of a model of \ZFC\ are uniformly definable by a $\Pi_2$-formula. (See also \cite[Section~2]{BagariaHamkinsTsaprounisUsuba2016:SuperstrongAndOtherLargeCardinalsAreNeverLaverIndestructible}.) That is, there is a $\Pi_2$-formula $\varphi(x,r)$ so that for each $r$ either $\{ x : \varphi(x,r) \}$ is empty or else it is a ground. 
So given the $\Pi_2$-elementary diagram of $M$ we can compute whether $\{ x : M \models \varphi(x,r) \}$ is nonempty, say by checking whether $\varphi(\emptyset,r)$ holds. We can thus compute the set $\{ (n,x) : M \models \varphi(x,r_n) \}$ where $r_n$ is the $n$th element $r$ of $M$, according to the order on $\omega$, so that $\{ x : \varphi(x,r) \}$ is nonempty. From this we can get a list of the $\Delta_0$-diagrams of the grounds of $M$.
\end{proof}

Let us highlight the assumption in the statement of this theorem that $M$ satisfies the axiom of choice, an assumption that was missing in the results in Section~\ref{Section.Computing-forcing-extensions}. In \cite{gitmanjohnstone:groundmodels}, Gitman and Johnstone showed for an ordinal $\delta$ that $\DC_\delta$, a version of the principle of dependent choice which is weaker than the full axiom of choice, suffices to establish that ground models are definable for a certain class of forcings, namely those with a gap at $\delta$. (See their paper for definitions and details.) They conjectured that the ground model definability theorem fails for \ZF. This remains an open problem, but Usuba has recently achieved some partial results \cite{usuba2019}. We assumed $M$ satisfies the axiom of choice because in this case we do know that the grounds are uniformly definable. If Gitman and Johnstone's conjecture were to be refuted, then we could improve this theorem to assume the model satisfies only \ZF\ instead of \ZFC.

\begin{corollary}\label{Corollary.Computing-grounds}
Given an oracle for the full elementary diagram of a model of set theory $\<M,\in^M> \models \ZFC$, there is a computable procedure to compute a list of the full elementary diagrams of the grounds of $M$.
\end{corollary}

\begin{proof}
This follows using the fact that the translation map on formulae $\varphi \mapsto \varphi^W$ is computable, if $W$ is a definable class, since the translation is merely replacing unbounded quantifiers with quantifiers bounded by $W$. So from the full elementary diagram of $M$ we can compute a listing of the elementary diagrams of the grounds $W$ of $M$.
\end{proof}

We turn now from the grounds to the full generic multiverse. It is not possible to compute a listing of all the models in the generic multiverse for the simple reason that the generic multiverse is uncountable. Even if we restrict to extensions from just a single simple forcing, for instance the forcing to add a Cohen real, there will still be uncountably many extensions. Nevertheless, the \emph{computable generic multiverse} of $M$, that portion of the multiverse computable from the diagram of $M$, is close to the full multiverse in a sense we now describe. 

Usuba's result that the grounds are strongly downward directed \cite{usuba2017} implies
that every model in the generic multiverse of $M$ is at most two steps away from $M$, namely it is a forcing extension of a ground of $M$. While we cannot in general hope that every model is computable from the elementary diagram of $M$, we can always compute a presentation of a model which is a forcing extension of the same ground by the same poset. Such models will necessarily satisfy the same set-theoretical formulae with parameters from $M$.

\begin{corollary}\label{Corollary.Computing-the-generic-multiverse}
Let $\langle\bar N,\in^{\bar N}\rangle$ be a model in the generic multiverse of $M$, where $\bar N = W[\bar G]$ for a distinguished ground $W$ of $M$ where $\bar G$ is generic over $W$ for a distinguished poset $\P$. Given an oracle for the full elementary diagram of $M$ there is a computable procedure to compute a $W$-generic filter $G \subseteq \P$ and decide the full elementary diagram of $N = W[G]$.
\end{corollary}

\begin{proof}
This follows immediately from Corollary~\ref{Corollary.Computing-grounds} and Theorem~\ref{Theorem.Computing-diagram-of-extension}.
\end{proof}

On the other hand, there is a different sense in which the computable generic multiverse of $M$ is far from the full generic multiverse of $M$. Namely, the computable generic multiverse is not dense in the generic multiverse. There are models in the generic multiverse so that no extension of them can be computed from the full diagram for $M$.

\begin{theorem}
Let $\<M,\in^M>$ be a countable model of \ZF. Then there is $M[G]$ a forcing extension of $M$ by the forcing to add a Cohen-generic $G \subseteq \omega^M$ so that no outer model of of $M[G]$ has a $\Delta_0$-diagram computable from the full elementary diagram of $M$. 
\end{theorem}

\begin{proof}
Let us first describe the generic $G$. Fix any real $z$, thought of as an $\omega$-length binary sequence, which is not computable from the full elementary diagram of $M$. From the diagram of $M$ we can compute a list of the dense subsets of $\Add(\omega,1)^M$. We use this list to build a generic, in the following manner. Start with $p_0 = \<z(0)>$. Having built $p_n$, first extend to meet the $n$th dense set, minimizing the length of the extension (if there is more than one minimal length extension to the $n$th dense set then pick arbitrarily). Then put $z(n+1)$ on the end to get $p_{n+1}$. Because we met every dense set, $G = \bigcup_n p_n$ is generic.\footnote{Note that this process works even if $M$ is $\omega$-nonstandard. In this case, we still have a list, whose order-type is the real $\omega$, of the dense subsets of $\Add(\omega,1)^M$. And since for each $a \in \omega^M$ it is dense in $\Add(\omega,1)^M$ to have a condition with length $\ge a$, the $G$ we produce is unbounded in $\omega^M$.}

Assume now that $\<N,\in^N>$ is an outer model of $M[G]$, where we think of the universe of $N$ as being $\N$ with $\in^N$ being some binary relation on $\N$. Let us see how to compute $z$ from the $\Delta_0$-diagram of $N$ and the full diagram for $M$. Without loss of generality we may assume that the ordinals $\le \omega$ in $M$ and $N$ are represented by the same natural numbers, as we may compute an isomorphic copy of $N$ with this property from what we are given. Fix the index of $G$ in $N$. From the $\Delta_0$-diagram of $N$ we can compute $z(0)$, simply by asking what the first bit of $G$ is. From the full diagram of $M$ we know the shortest distance we have to extend past $p_0 = \<z(0)>$ to meet the $0$th dense set. So we can compute the next coding point and thereby recover $z(1)$ and $p_1$. Continuing this process upward, we can compute $z(n)$ for each $n$. Therefore, if we could compute the $\Delta_0$-diagram of $N$ from the full diagram of $M$ then we could compute $z$ from the full diagram of $M$, which would be a contradiction.
\end{proof}

Next we wish to discuss the extent to which the computable generic multiverse has the same structural properties as the full generic multiverse. Let us start with the following property, due essentially to Mostowski \cite{Mostowski1976}. Given a collection $\mathcal E$ of models in the generic multiverse, say that $\mathcal E$ is \emph{amalgamable} when there is a model in the generic multiverse which contains every model in $\mathcal E$. Note that if each model in $\mathcal E$ is a forcing extension of $M$ it is equivalent to ask whether there is a forcing extension of $M$ which contains every model in $\mathcal E$. 

\begin{theorem}[Mostowski]
Let $I$ be a finite set and let $\mathcal A$ be a family of subsets of $I$ which contains all singletons and is closed under subsets. Let $\<M,\in^M>\satisfies\ZF$ be a countable  model of set theory. Then there are reals $c_i \subseteq \omega^M$ for $i \in I$ so that each $c_i$ is Cohen-generic over $M$ and for $A \subseteq I$ the family $\{M[c_i] : i \in A\}$ is amalgamable if and only if $A \in \mathcal A$. 
\end{theorem}

The same phenomenon happens within the computable generic multiverse.

\begin{theorem}
Let $I$ be a finite set and let $\mathcal A$ be a family of subsets of $I$ which contains all singletons and is closed under subsets. Let $\<M,\in^M>\satisfies\ZF$ be a countable model of set theory. Then from an oracle for the elementary diagram of $M$ there is a procedure to compute Cohen reals $c_i \subseteq \omega^M$ generic over $M$ and the elementary diagrams of $M[c_i]$ for $i \in I$ so that for $A \subseteq I$ the family $\{M[c_i] : i \in A\}$ is amalgamable in the generic multiverse if and only if $A \in \mathcal A$. 
\end{theorem}

\begin{proof}
In the language of \cite{HHKVW2019}, let $z$ be a \emph{catastrophic real} for $M$; that is, $z$ is a real so that no outer model of $M$ can contain $z$. We claim that there is such $z$ computable from the elementary diagram of $M$. Namely, we can take $z$ to be an isomorphic copy of $\in^M$ on $\omega$, along with with an isomorphism onto $M$. Then, by the Mostowski collapse lemma, any model of \ZF\ which contains $z$ would have to contain $\in^M$ itself as a set, which is impossible for an outer model of $M$.

Without loss of generality we may assume that $I,\mathcal A \in M$. For each $A \subseteq I$ let $\P_A$ be the forcing $\prod_{i \in A} \Add(\omega,1) \in M$. 
From the elementary diagram of $M$ we can compute a listing, in order-type $\omega$, of all pairs $\<A,D>$ with $A \in \mathcal A$ and $D \in M$ a dense subset of $\P_A$. We build the Cohen reals $c_i$ by means of a descending sequence of conditions, which we think of as filling in an $\omega^M \times I$ matrix with $0$s and $1$s, with the $i$th column growing into $c_i$. We will ensure that at each step of the construction we have built all columns up to the same height. 

We start with a completely empty matrix, i.e.\ with $c^0_i = \emptyset$ for all $i \in I$. Now suppose we have built up $c^n_i$. We are presented with $A_n \in \mathcal A$ and $D_n \subseteq \P_A$ dense. Extend the columns with index in $A_n$ to collectively meet $D_n$, then pad with $1$s to ensure the columns all have the same height. Next, pad the remaining columns with $0$s to build them up to the same height, then extend each column by appending a row of $1$s followed by a row of $z(n)$'s. These rows of $1$s are the coding points which will be used to recover $z(n)$ if we are dealing with $A \not \in \mathcal A$. Note that this process is computable given the $\Delta_0$-diagram of $M$, since from that we can compute the minimal length we need to extend to meet $D_n$ and then pick one of the finitely many extensions of that length. So if we set $c_i$ to be the generic determined by $\<c^n_i : n \in \omega>$ then $c_i$ is computable from the diagram of $M$. And so, once we know the $c_i$'s are generic we know that we can, as before, compute the full diagrams of the $M[c_i]$'s. 

It remains only to see that the $M[c_i]$'s have the desired amalgamability property. First, suppose that $A \in \mathcal A$. Then we built up $\{ c_i : i \in A \}$ so that they met every dense subset of $\P_A$ in $M$. So the $c_i$ for $i \in A$ are mutually generic and so the family $\{ M[c_i] : i \in A \}$ is amalgamable, witnessed by $M[c_i : i \in A]$. In particular, this shows that each $c_i$ is generic over $M$. Now suppose that $A \in \mathcal P(I) \setminus \mathcal A$. Then, by the construction, the only rows in which each $c_i$ for $i \in A$ has value $1$ are the coding points identifying where the bits of $z$ are coded. So no outer model of $M$ which satisfies $\ZFC$ can contain each $c_i$ for $i \in A$.
\end{proof}

In \cite{HHKVW2019}, the first and third author along with Habi\v{c}, Klausner, and Verner extended Mostowski's theorem. We will not reproduce that article here, but we wish to note that the constructions therein are all effective. Given an oracle for the elementary diagram of $M$ there are computable procedures to compute the desired generics for the results from that article. So the properties of the generic multiverse explored in that article are also enjoyed by the computable generic multiverse.
\smallskip 

To close out this section, we remark that the existence of many (non-isomorphic) grounds for the same countable
model of \ZFC\ implies that in general it is impossible to recover $M$ effectively---or even non-effectively---from an arbitrary copy of $M[G]$.  $M$ has a canonical embedding into $M[G]$
by the map sending each $x\in M$ to the class $[\xcheck]_G$, where the $\P$-name
$\xcheck$ is defined by recursion as $\set{\<\ycheck,1>~ : y\in x}$.  With a
$\Delta_0(M)$-oracle, one can compute $\xcheck$ from $x$, by the methods
seen earlier for recursive definitions,
and thereby compute the canonical embedding of $M$ into $M[G]$.
Its image will thus be $\Delta_0(M)$-computably enumerable in $M[G]$,
but it is not defined uniformly across copies of $M[G]$.

The question of how to recover a copy of $M$
from a copy of $M[G]$ is closely tied to the question of whether
$\Delta_0(M[G])$ can compute $\Delta_0(M)$, and if so, whether there
is a uniform procedure for doing so from all copies of $M[G]$.  The answer is not obvious.
Indeed, one can imagine the possibility that the isomorphism type of the structure
$M[G]$ may be simpler in some sense than that of $M$, and that therefore
there may exist a copy of $M[G]$ that cannot compute any copy of $M$.
For example, perhaps $M[G]$ satisfies $\GCH$, whereas the map
$\kappa\mapsto 2^{\kappa}$ on cardinals in $M$ may have been far more
chaotic and may have encoded some information not intrinsically recoverable
from $M[G]$.

There is an analogy here to computable fields.  Rabin's Theorem states that for
every computable field $F$, the algebraic closure $\Fbar$ also has a computable
presentation, and that both a copy of $\Fbar$ and an embedding of $F$ into that copy
may be computed uniformly from the atomic diagram of $F$, in the signature
with $+$ and~$\cdot$.  This much is analogous to our results above for a given
$(M,\P)$.  However, the uniformity carries over to countable fields $F$ that
are not computably presentable, and in this case the algebraic closure may be
far simpler than any presentation of $F$, as every countable algebraically closed field
has a computable presentation.  Taking the algebraic closure smoothes out a field
and eliminates complexity, and we ask whether the same might happen with a forcing extension of a model of set theory. 

\begin{question}
Let $M$ be a countable model of \ZFC, and $\P\in M$ a forcing notion, for which the
procedure in Theorem  \ref{Theorem.Computing-atomic-diagram-of-extension} computes
a filter $G$ and the atomic diagram of a presentation of $M[G]$.  Can there exist a
presentation $\mathcal A\cong M[G]$ such that for every presentation ${M^*} \cong M$, we have $\Delta_0({M^*}) \not \leq_T \Delta_0(\mathcal A)$?  And if so, can this presentation $\mathcal A$ be the one computed by our procedure?
\end{question}

\section{Class forcing}\label{Section.Class-forcing}

Elsewhere, we have restricted our attention to forcing notions which are set-sized. We detour in this section to consider proper-class-sized forcing notions. There are two major approaches to formulate class forcing, and we consider both of them.
The first approach, let us call it the first-order approach, is to work over \ZF\  (possibly assuming more) and deal with a definable class. A generic then has to meet every definable dense subclass of the forcing notion.

\begin{theorem}
Given an oracle for the full elementary diagram of a countable model $\<M,\in^M> \models \ZF$ and given a definable pretame class forcing $\P \subseteq M$, there is a computable procedure to compute an $M$-generic filter $G \subseteq \P$ and decide the full elementary diagram of the forcing extension $M[G]$.
\end{theorem}

See \cite[Section 2.2]{friedman:book} for a definition of pretameness, which is equivalent to the preservation of $\ZF^-$ (in the language with a predicate for the generic filter). The reason to ask $\P$ to be pretame is that the pretame forcings are precisely those which have a definable atomic forcing relation \cite{HKS2018}.

\begin{proof}
From the full diagram of $M$ we can compute a list of the definable dense subclasses of $\P$. So we can compute $G$ as in Theorem~\ref{Theorem.Computing-generic-filter}. Now given the atomic forcing relation for $\P$ there is a computable procedure to associate a formula $\varphi$ with the formula defining the corresponding forcing relation $p \forces \varphi$. So we can compute $M[G]$ and its full elementary diagram as in Theorem~\ref{Theorem.Computing-diagram-of-extension}.
\end{proof}

This answers the question for the first-order approach. We turn now to the other approach, call it the second-order approach. For this approach classes are actual objects in our models, where we work (in first-order logic) with two-sorted structures. We work over a second-order set theory, such as G\"odel--Bernays set theory \GB{} or Kelley--Morse set theory \KM.\footnote{\GB{} is the weaker of the two, stating the existence of classes defined predicatively---quantification is allowed only over sets. On the other hand, \KM{} allows for impredicative comprehension, defining classes by quantifying over the classes. See \cite[Section~2]{williams-min-km} for precise axiomatizations of these two theories, as well as a discussion for their place in the hierarchies of second-order set theories.}
We will use italic letters such $M$ to refer to the sets and calligraphic letters such as $\mathcal M$ to refer to the classes of a model of second-order set theory. Abusing notation slightly, we will also use $\mathcal M$ to refer to the whole model; this is unambiguous, as the sets are definable from the classes. A class forcing notion $\P$ is then a class in the model and a generic meets every dense subclass of $\P$ in the model. We will write $\mathcal M[G]$ for the extension by a generic $G \subseteq \P$.\footnote{Our approach is in first-order logic with two-sorted structures, but as is well-known one can equivalently work in second-order logic with Henkin semantics. In this semantics, one explicitly lists out the classes considered for the second-order part of semantics. Note, however, that to make this approach amenable to the context of computable structure theory we would have to attach to each model a list of the classes to be used for its second-order semantics. That is, we would need natural numbers to identify each class and include a relation for the set-class membership relation. This amounts to the same as the approach in first-order logic.}

Let $\Delta^1_0$ refer to the class of formulae with only set quantifiers. Up to equivalence, this is the same as the class of formulae where all quantifiers are bounded, possibly by classes, because $\exists x\ \varphi(x)$ is equivalent to $\exists x \in V\ \varphi(x)$. So this is the second-order analogue of $\Delta_0$ in the first-order \Levy{} hierarchy.

\begin{theorem}\label{Theorem.Computing-first-order-diagram-for-class-forcing}
Let $\<\mathcal{M},\in^{\mathcal{M}}>$ be a countable model of \GB{} and suppose $\P \in \mathcal M$ is a class forcing notion with its atomic forcing relation a class in $\mathcal M$. Then, from an oracle for the $\Delta^1_0$-elementary diagram of $\mathcal{M}$ there is a computable procedure to compute an $\mathcal{M}$-generic filter $G \subseteq \P$ and the $\Delta^1_0$-diagram of $\mathcal{M}[G]$. 
\end{theorem}

\begin{proof}
  Begin by observing that it is $\Delta^1_0$ to say that a class is a dense subclass of $\P$. So from the $\Delta^1_0$-diagram for $\mathcal M$ we can compute a list of all the dense subclasses of $\P$. We can then compute an $\mathcal M$-generic filter $G \subseteq \P$ as in Theorem~\ref{Theorem.Computing-generic-filter}. Next, note that being a class $\P$-name is a $\Delta^1_0$ property, because a class $\P$-name is a class whose elements are all pairs of elements of $\P$ and a set $\P$-name, and being a set $\P$-name is a $\Delta^1_0$ property. So we can decide which elements of $\mathcal M$ are $\P$-names. Further observe that the relations
$$
p \forces  \sigma \in \tau, \quad p \forces \sigma = \tau, \quad p \forces \sigma \in T, \quad p \forces \Sigma = T
$$
are all $\Delta^1_0$. As in the proof of Theorem~\ref{Theorem.Computing-diagram-of-extension} given the generic $G$ we can define relations $=_G$ and $\in_G$ on the $\P$-names, which are decidable from the $\Delta^1_0$-diagram. And so we can pick out representatives of the $=_G$-equivalence classes, thereby computing the atomic diagram of $\mathcal M[G]$. Finally, observe that for a $\Delta^1_0$-formula $\varphi$ that $p \forces \varphi(\Sigma_0, \ldots, \Sigma_n)$ is $\Delta^1_0$, as it is defined from the atomic forcing relation by quantifying over sets. So we can decide the $\Delta^1_0$-diagram of $\mathcal M[G]$.
\end{proof}

As in the set forcing case, from the full elementary diagram of $\mathcal M$ we can compute the full elementary diagram of the class forcing extension. 

\begin{theorem}
Let $\<\mathcal{M},\in^{\mathcal{M}}>$ be a countable model of $\GB{}$ and suppose $\P \in \mathcal M$ is a class forcing notion with its atomic forcing relation in $\mathcal M$. Then, from an oracle for the second-order elementary diagram of $\mathcal M$ there is a computable procedure to compute an $\mathcal M$-generic filter $G \subseteq \P$ and the second-order elementary diagram of $\mathcal{M}[G]$. \end{theorem}

To clarify, by the \emph{second-order elementary diagram} of a model of second-order set theory we mean second-order in the sense of allowing quantifying over the classes of the model---that is, in the Henkin semantics---not second-order in some external sense. 

\begin{proof}
By Theorem~\ref{Theorem.Computing-first-order-diagram-for-class-forcing} we can compute $G$ and the $\Delta^1_0$-diagram of $\mathcal M[G]$. The result for the full diagram then follows from the fact that $p \forces \varphi(T_0, \ldots, T_n)$ is second-order definable when $\varphi$ is second-order. 
\end{proof}

It is natural to ask whether this goes level-by-level. The $\Sigma^1_n$ and $\Pi^1_n$-formulae are inductively defined from the $\Delta^1_0$-formulae similar to how in first-order set theory the $\Sigma_n$ and $\Pi_n$-formulae are defined from the $\Delta_0$-formulae. For instance, a formula is $\Sigma^1_2$ if it is of the form $\exists X \forall Y\ \varphi(X,Y)$, where both quantifiers are over the classes and $\varphi$ is $\Delta^1_0$. To argue this goes level-by-level as in the set forcing case we would need that if $\varphi$ is $\Sigma^1_n$ then $p \forces \varphi$ is $\Sigma^1_n$. To prove the analogous fact for the set forcing case we used the Replacement schema.\footnote{Cf. Footnote \ref{Footnote.Replacement}.} The same argument works in the class forcing case if our model satisfies the Class Collection schema, a second-order version of the Replacement schema.

\begin{definition}
The \emph{Class Collection} axiom schema asserts that if for every set there is a class satisfying some property, then there is a single class coding the ``meta-class'' consisting of a witnessing class for every set. Formally, instances of this schema take the form
$$\forall \bar P \left[ (\forall x \exists Y\ \varphi(x,Y,\bar P)) \Rightarrow (\exists C \forall x\ \exists i\  \varphi(x,(C)_i,\bar P)) \right],$$
where $(C)_i = \{ y : (i,y) \in C \}$.\footnote{Observe that in the presence of Global Choice---the assertion that every class can be well-ordered---we may equivalently ask that this index $i$ for $x$ be $x$ itself.}
For $n \in \omega$, the \emph{$\Sigma^1_n$-Class Collection schema} is the restriction of Class Collection to $\Sigma^1_n$-formulae.
\end{definition}

It is simple to check that $\Sigma^1_n$-Class Collection implies $\Sigma^1_n$-Comprehension. The converse does not hold. Even $\KM$ with the full impredicative Comprehension schema cannot prove $\Sigma^1_0$-Class Collection \cite{GitmanHamkinsKaragila:KM-set-theory-does-not-prove-the-class-Fodor-theorem}.
However, adding Class Collection does not increase consistency strength---see \cite[Theorem 2.5]{marek-mostowski1975} for the $\KM$/full Class Collection case and \cite{ratajczyk1979} for the level-by-level case.

\begin{corollary}
Let $\<\mathcal{M},\in^{\mathcal{M}}>$ be a countable model of $\GB{}$ $+$ $\Sigma^1_n$-Class Collection and suppose $\P \in \mathcal M$ is a class forcing notion with its atomic forcing relation in $\mathcal M$. Then, from an oracle for the $\Sigma^1_n$-elementary diagram of $\mathcal M$ there is a computable procedure to produce an $\mathcal M$-generic filter $G \subseteq \P$ and the  $\Sigma^1_n$-elementary diagram of $\mathcal{M}[G]$. \qed
\end{corollary}

\section{Functoriality and Interpretability}
\label{sec:functoriality}

In this article we are considering an effective procedure mapping one class of models
on the domain $\omega$ to another such class (in fact, to the same class). In this section we recall some known theorems about this scenario and what it says about effective interpretability, relating these general facts to our specific case of a model $M$ of set theory inside a forcing extension $M[G]$. 

Theorems from \cite{HTM3} and \cite{HTM2} relate such procedures
to the interpretability of each output model (in the second class) in the
corresponding input model (in the first class).  Here we repeat
the simplest versions of those theorems.  The gist is that interpretations
of one structure in another by $\mathcal{L}_{\omega_1,\omega}$-formulae
(and with no fixed arity on the domain of the interpretation) correspond bijectively
to functors from the category of isomorphic copies of the second structure
(with isomorphisms as the morphisms in the category)
into the category of isomorphic copies of the first.
Rather than attempt to define all the terms here, we refer the reader to
\cite[Definition 1.2]{HTM3}, \cite[Definition 2.1]{HTM2}, and \cite[Definition 5.1]{M14} for the notions of interpretability, and to \cite[Definition 1.2]{HTM3}, \cite[Definition 2.7]{HTM2}, and \cite[Definition 3.1]{MPSS} for the notions about functors.

\begin{theorem}[Theorem 1.5 of \cite{HTM3}]
\label{thm:HTM3}
Let $\mathcal A$ and $\mathcal B$ be countable structures. Then
$\mathcal A$ is effectively interpretable in $\mathcal B$ if and only if
there exists a computable functor from the category $\mathrm{Iso}(\mathcal B)$
of isomorphic copies of $\mathcal B$ (under isomorphism) to the
corresponding category $\mathrm{Iso}(\mathcal A)$.
\end{theorem}

\begin{theorem}[Theorem 2.9 of \cite{HTM2}]
\label{thm:HTM2}
Let $\mathcal B$ and $\mathcal A$ be countable structures, possibly in different countable languages.
For each Baire-measurable functor $F:\mathrm{Iso}(\mathcal B)\to \mathrm{Iso}(\mathcal A)$ there is an
infinitary interpretation $\mathcal I$ of $\mathcal A$ within $\mathcal B$, such that $F$ is naturally isomorphic
to the functor $F_{\mathcal I}$ associated to $\mathcal I$.  Furthermore, if $F$ is $\Delta^0_{\alpha}$
in the lightface Borel hierarchy, then the interpretation can be taken to use
computable $\Delta^0_\alpha$ $\mathcal{L}_{\omega_1,\omega}$-formulae and the isomorphism
between $F$ and $F_{\mathcal I}$ can be taken to be $\Delta^0_\alpha$.
\end{theorem}

Those sources also extend these theorems to situations of bi-interpretability
between countable structures, and (of greater relevance here)
to interpretations that hold uniformly from one class of structures to another,
rather than interpretations merely of $\mathcal A$ in $\mathcal B$.  For us, the point is simply
that results about functors and their complexity, such as the simple ones
we derive in this section, correspond to results about interpretability.
For example, a computable functor mapping copies of $M$ to copies
of $M[G]$ would correspond to an effective interpretation of $M[G]$ in $M$.

A computable functor is a particularly strong kind of operator on countable structures.
Characteristically, attempting to create such an operator requires one to determine
exactly which aspects of the structure are relevant to the operation.
The following analysis is not difficult, but it serves as a good example of this principle.

We here extend the \Levy\ signature defined above to a larger signature.
Along with the symbols from the former, this new signature has countably many
constants $p$, $c$, and $d_0,d_1,\ldots$.  The intention (which can be expressed
as an $\mathcal{L}_{\omega_1,\omega}$-formula, but not by any finitary axiom) is that 
$p$ denotes the partial order $\P$ used for forcing in this structure,
that the constants $d_j$ name precisely the dense subsets of $\P$,
and that $c$ names a choice function on the nonempty subsets of $\P$ in $M$.
We consider models of \ZFC\ here, not just \ZF, to ensure that $M$ will
contain such a choice function, although in many cases $\P$ will
have a readily definable choice function.  It is important
that $c$ be internal to $M$, and so formally we treat $c$ as a constant symbol,
rather than a function symbol.
The domain of our functor is the category of models of \ZFC\ on the domain $\omega$,
in our larger signature, satisfying these conditions.
The morphisms between two such structures are exactly the isomorphisms of structures
in this signature.  The range of the functor is the category of models of \ZFC\ on the
domain $\omega$ in the \Levy\ signature, since in the forcing extension,
the forcing notion is no longer relevant.  (If desired, one can use the larger signature, with the check-name of $p^M$ as $p^{M[G]}$ and similarly for $c$ and the $d_j$.)

The construction of the generic $G$ is now done effectively by a prescribed method, which is
more precise than the method of Theorem \ref{Theorem.Computing-atomic-diagram-of-extension}.
The added precision is necessary to allow the functor to be computed with no more constants than these.
In the new method, the functional $\Phi$ is given as an oracle the atomic diagram of $M$
in the larger signature, namely $\Delta_0(M)\oplus\langle p^M,c^M,d_0^M,d_1^M,\ldots\rangle$,
and searches first for the maximum element $p_0$ of the forcing notion $\P$
named by the constant $p$.  The $\Delta_0$-diagram identifies it, as it is defined by
$p_0\in\P \land (\forall q\in\P) q\leq_{\P}p_0$.  Next, for each $s$ in turn, we find the
element $d_s^M$ and use $c$ to select an extension $p_{s+1}\leq p_s$
in $d_s\cap\P$. We can do so because the set $x = \{ q \in^M d_s : q \le p_s \} \in M$ is $\Delta_0$-definable from parameters we have access to. So we can search for the element of $M$ which satisfies this defining property and know that the search must terminate. Then we can search for $c(x) \in M$ and choose it to be $p_{s+1}$.
While we use the $<$ relation on the domain $\omega$ of $M$ to search for $x$ and then $c(x)$, note that this procedure is nevertheless uniform as the search is looking for a uniquely determined object.

If $f:M \to M^*$ is an isomorphism (in the larger signature),
then the same procedure $\Phi$ on $\Delta_0(M^*,p^{M^*},\vec{d}^{\,M^*},c^{M^*})$
will clearly produce a forcing extension $M^*[G^*]$ isomorphic to $M[G]$.
The point of this method is that there is a Turing functional $\Phi_*$ which
from the oracle
$$\Delta_0(M,p^M,\vec{d}^{\,M},c^M)\oplus f \oplus \Delta_0(M^*,p^{M^*},\vec{d}^{\,M^*},c^{M^*})$$
computes the corresponding isomorphism from $M[G]$ onto $M^*[G^*]$.
This map on isomorphisms will respect composition and, if given the identity
isomorphism on $M$, will produce the identity isomorphism on $M[G]$.
We do not consider it necessary to write out the proof of this result.
This $\Phi_*$, together with the functional $\Phi$ that computes $\Delta_0(M[G]) = \Phi^{\Delta_0(M)}$,
constitutes a program for the forcing functor, showing that
$\mathcal{F}$ is in fact a computable functor, according to the definition in \cite{MPSS}.

We remark additionally that the forcing functor works exactly the same way
in the category with the same objects as above, but in which the morphisms now include
all injective homomorphisms from any one forcing structure to any other.
Of course, the image of a morphism $g$ now need only be an injective homomorphism itself,
not an isomorphism, but functoriality still holds, effectively.

None of the foregoing is at all difficult. With the new constant symbols enumerating the dense sets we get functoriality but, as will be established in the next section, if we omit them we do not get functoriality.
Meanwhile, though, we can apply Theorem \ref{thm:HTM3} to our computable functor and extract
from it an effective interpretation, uniformly across all forcing extensions built by the functor.

\begin{corollary}
\label{cor:interpretations}
There exist fixed computable infinitary $\Sigma_1$-formulae that,
for every ground model
$(M, \Delta_0(M),p,c,d_0,d_1,\ldots)$ of \ZFC\ in the larger signature,
give an effective interpretation of the forcing extension $M[G]$
in $(M, \Delta_0(M),p,c,d_0,d_1,\ldots)$, provided
$G$ is built from this presentation as described in this section.
(The interpretation allows tuples from $M$ of arbitrary finite arity in its domain.)
\end{corollary}

\section{Non-functoriality}\label{Section.Non-functoriality}

In the main theorem, we proved that there is a computable procedure 
 $$(M,\in^M,\P)\mapsto G$$
that takes as input the atomic diagram of a model of set theory $\<M,\in^M,\P>$ with a distinguished notion of forcing $\P\in M$, and produces an $M$-generic filter $G\of\P$. We had observed, however, that the particular filter $G$ arising from this process can depend on how exactly the atomic diagram of $M$ is presented to us. If we rearrange the copy of $M$ as it is represented on the natural numbers (to be used as an oracle for the computation), then this can affect the order in which the dense sets $d_s$ appear and therefore affect which conditions $p_s$ arise in the construction. In short, the computational procedure we provided does not respect isomorphisms of presentations, since different isomorphic presentations of the same model can lead to non-isomorphic generic filters. Thus the process we described for computing the generic filter is not functorial in the category of presentations of models of set theory under isomorphism.

We claimed in Section \ref{sec:functoriality} that in this signature this phenomenon is unavoidable. To justify that claim, we now show that even if one allows the full elementary diagram of the model $M$ as input, there is no effective procedure that will pick out the same generic filter in all presentations of the same model.

\begin{theorem}\label{Theorem.No-Turing-functor}
 If \ZF\ is consistent, then there is no computable procedure that takes as input the elementary diagram of a model of set theory $\<M,\in^M,\P>$ with a partial order $\P$ and produces an $M$-generic filter $G\of \P$, such that isomorphic copies of the input model result always in the same corresponding isomorphic copy of $G$.
\end{theorem}

\noindent In other words, there is no computable procedure to produce generic filters that is functorial in the category of presentations of models of set theory under isomorphism.

\begin{proof}
Assume toward contradiction that we have a computable procedure 
 $$\Phi:\Delta(M,\in^M,\P)\quad\mapsto\quad G,$$ 
where we assume $M=\N$ and $\<M,\in^M,\P>\satisfies\ZF$ and $\P\in M$ is a partial order in $M$, such that $G\of\P$ is $M$-generic and this process is functorial, in the sense that isomorphic presentations of $\<M,\in^M>$ lead always to the same isomorphic copy of $G$.

We assumed that \ZF\ is consistent. It follows by the \Levy--Montague reflection principle and a simple compactness argument that there is a countable model $M\satisfies\ZF$ such that $M_\kappa\elesub M$ for some cardinal $\kappa$ in $M$, where $M_\kappa$ means the rank-initial segment $(V_\kappa)^M$.\footnote{In general, $M$ may be $\omega$-nonstandard. But note that the existence of well-founded $M$ with $M_\kappa \elesub M$ is consistent relative an inaccessible cardinal (indeed, less is necessary). However, this assumption is higher in consistency strength than just asking for a well-founded model of set theory.}
Fix any nontrivial forcing notion $\P\in M_\kappa$, such as the forcing to add a Cohen real. 

The main idea of the proof is to try to run the computational process inside the model $M$. This doesn't make literal sense, of course, since in $M$ we do not have the elementary diagram of $M$ as a countable set. Nevertheless, in $M$ we do have the elementary diagram of $M_\kappa$, which although uncountable, is a set in $M$ and therefore has its diagram in $M$. It may be that $M$ is $\omega$-nonstandard and therefore may also have nonstandard-length formulae in its version of the diagram, which of course cannot be part of the real elementary diagram. But on the standard-length formulae, $M$ will agree with us on the elementary diagram of $M_\kappa$.

Inside $M$, consider all the ways that we might place finitely much information about the elementary diagram of $M_\kappa$. Specifically, we enumerate finitely many elements of $M_\kappa$ and then list off finitely many truth judgments about those elements, and use this as a partial oracle in the computational process of $\Phi$. 

If outside $M$ we were to consider a full presentation of the elementary diagram of $M_\kappa$, then the process $\Phi$ will result in complete judgments about membership in an $M_\kappa$-generic filter $G\of\P$, and furthermore these judgments will be independent of the presentation of the diagram we consider. It follows that for any particular condition $p\in\P$, there is a finite piece of the (full, actual) presentation that leads to a judgment by $\Phi$ either that $p\in G$ or that $p\notin G$. 

The key observation is that this finite piece of the full actual diagram of $M_\kappa$ will be an element of $M$ and therefore $M$ will be able to see the judgment that $\Phi$ makes on whether $p\in G$ or not. Furthermore, all minimal-length such pieces of the diagram of $M_\kappa$ in $M$ that decide whether $p\in G$ or not will agree on the outcome, since any such minimal-length piece will involve only standard-finite formulae and therefore in principle can be continued outside $M$ to a presentation of the full actual diagram of $M_\kappa$, which always gives the same result about membership in $G$. Consequently, by searching inside $M$ for these minimal-length supporting computations about $M_\kappa$, we conclude that $G\in M$. 

But this is impossible, since all subsets of $\P$ in $M$ are in $M_\kappa$, since $M_\kappa$ is a rank-initial segment of $M$, and so $G$ will actually be $M$-generic, while also an element of $M$, which is impossible for nontrivial forcings.
\end{proof}

It follows that, given suitable consistency assumptions, there also can be no functorial computable procedure for producing just the atomic diagram of the forcing extension $M[G]$, since there are some forcing notions $\P$, such as the self-encoding forcing defined in \cite{FuchsHamkinsReitz2015:Set-theoreticGeology}, that produce unique generic filters in their extensions; for such forcing notions over a well-founded ground model, the extensions are isomorphic if and only if the generic filters are identical.

In the argument for the proof of Theorem~\ref{Theorem.No-Turing-functor} we produced a specific model which witnessed the failure of a would-be Turing functor $\Phi$. It is natural to wonder whether given any $\Phi$ and given any countable model $M$ of set theory we can witness the failure of $\Phi$ with an isomorphic copy of $M$.

The following theorem answers this question in the negative. If we restrict to the pointwise definable models then we can produce generics in a functorial manner. A model of set theory is \emph{pointwise definable} if every element of the model is definable without parameters. For example, the Shepherdson--Cohen minimum transitive model of \ZF\ is pointwise definable.

\begin{theorem}
There is a computable functor $\Phi$, in which $\Phi$ takes as input the elementary diagram of any pointwise definable model $\<M,\in^M> \satisfies \ZFC$ and a forcing notion $\P \in M$ and returns an $M$-generic $G \of \P$ and the elementary diagram of $M[G]$. That is, if $\<M^*,\in^{M^*}>$ and $\<M^\dagger,\in^{M^\dagger}>$ are two isomorphic presentations of $M$ then $\Phi(M^*,\in^{M^*},\P^*) \cong \Phi(M^\dagger,\in^{M^\dagger},\P^\dagger)$.
\end{theorem}

\begin{proof}
The key step in getting functoriality is to ensure that isomorphic presentations give rise to corresponding isomorphic generic filters. That is, if $\pi : M^* \to M^\dagger$ is an isomorphism we want that $\pi\image G^* = G^\dagger$, where $G^*$ and $G^\dagger$ are the generics produced by $\Phi$. Given this it is then straightforward that the isomorphism between $M^*$ and $M^\dagger$ extends to an isomorphism between $M^*[G^*]$ and $M^\dagger[G^\dagger]$. 

Let us see how to produce the generic. There is a canonical computable listing of the possible definitions in the language of set theory in order-type $\omega$. From the elementary diagram of $M$ we can decide which element is defined by which definition. So using the fact that $M$ is pointwise definable we can produce a listing $m_0, m_1, \ldots$ of the elements of $M$, starting with the element defined by the zeroth definition, then the element defined by the first definition, and so on. This listing is canonical, in the sense that isomorphic copies of $M$ will give rise to the same listing. More formally: if $\pi : M^* \to M^\dagger$ is an isomorphism then $\pi(m_i^*) = m_i^\dagger$ for all $i \in \omega$. We then use this listing to list out the dense subsets of $\P$ and to define the generic, as in the argument for Theorem~\ref{Theorem.Computing-generic-filter}. Because isomorphic copies of $M$ will use the same listing of their elements, this process will produce the same generic filter. 
\end{proof}

We claim that the non-functoriality result extends beyond computability to the case of Borel processes. Note that in the Borel context, from the atomic diagram of a model we can recover its full elementary diagram by a Borel function, and so the distinction between those cases evaporates.

\begin{theorem}\label{thm:Borel}
Suppose \ZF\ is consistent. Then there is no Borel function 
  $$(M,\in^M,\P)\quad\mapsto\quad G$$
 mapping codes for countable $\<M,\in^M,\P> \models \ZF$ with a forcing notion $\P\in M$ to an $M$-generic filter $G\of\P$, such that isomorphic models lead always to the same (isomorphic) filter. Indeed, we cannot even get such a Borel function so that if $\<M^*,\in^{M^*},\P^*>$ and $\<M^\dagger,\in^{M^\dagger},\P^\dagger>$ are elementarily equivalent then so are $\<M^*[G^*],\in^{M^*[G^*]}>$ and $\<M^\dagger[G^\dagger],\in^{M^\dagger[G^\dagger]}>$. 
\end{theorem}

We first proved this result with a mild extra consistency assumption. Philipp Schlicht privately communicated to us an argument which needs only the minimum assumption that \ZF\ is consistent. With his permission, it is his argument we present here.

\begin{proof}[Proof (Schlicht)]
Suppose that there is a Borel function $\Phi(M,\in^M,\P)=G$, defined from some real parameter $y$, such that if $\<M^*,\in^{M^*},\P^*>$ and $\<M^\dagger,\in^{M^\dagger},\P^\dagger>$ are codes for elementarily equivalent countable models of \ZF\ (coded, say, as a binary relation $\in^M$ on the natural numbers $M=\N$) equipped with forcing notions $\P^*\in M^*$ and $\P^\dagger \in M^\dagger$, then if $G^* = \Phi(M^*,\in^{M^*},\P^*)$ and $G^\dagger = \Phi(M^\dagger,\in^{M^\dagger},\P^\dagger)$ then $\<M^*[G^*],\in^{M^*[G^*]}>$ and $\<M^\dagger[G^\dagger],\in^{M^\dagger[G^\dagger]}>$ are elementarily equivalent. In particular, if $\<M^*,\in^{M^*}>$ and $\<M^\dagger,\in^{M^\dagger}>$ are isomorphic, then $\Phi$ produces elementarily equivalent forcing extensions. 

The desired counterexample will be a pair of isomorphic presentations of a Cohen-generic extension of a pointwise definable model of \ZF\ + $V = L$. Observe that the existence of such follows from the assumption that \ZF\ is consistent: A famous result due to G\"odel gives us a model of \ZF\ + $V = L$ and using that model's definable global well-order we get Skolem functions so that the Skolem hull of the empty set is pointwise definable. Call this pointwise definable model \hbox{$\<N,\in^N>$}.

Work in $N$. Let $\Q \in N$ denote the Cohen forcing poset to add a single real (in the sense of $N$) and let $\rho \in N$ be a $\Q$-name for the set of finite variants of the generic real. Consider $\dot\P$ a $\Q$-name so that $\one_\Q$ forces that $\dot\P$ is the poset which chooses an element of $\rho$ by lottery and codes that real into the continuum pattern below $\aleph_\omega$, via adding Cohen-generic subsets to the $\aleph_n$.\footnote{Given a collection of forcing notions $\P_i$, their \emph{lottery sum} is the poset which contains the $\P_i$ as suborders, conditions from different $\P_i$ being incomparable, and with a new maximum element placed above each $\P_i$. A generic for the lottery sum then chooses one of the $\P_i$ by lottery and produces a generic for it. See \cite[Section 3]{Hamkins2000:LotteryPreparation} for a precise definition.}
 Observe that if the generic real is modified on a finite domain, then this forcing does not change. 

In $V$, let $H$ be the transitive collapse of a countable elementary submodel of some large enough $H_\theta$ which contains $y$ and a code for $N$. Observe that we can think of $\Q$ as a forcing poset in $H$. Of course, $N$ may be $\omega$-nonstandard and so Cohen forcing in the sense of $N$ need not be the real Cohen forcing. Nevertheless, $H$ sees that $\Q$ is a poset, as that is absolute even for nonstandard models of set theory. Now let $x$ be $H$-generic for $\Q$. Then $x$ must also be $N$-generic for $\Q$. Set $M = N[x]$ and $\P = \dot\P^x \in M$. Then $\<M,\in^M,\P>$ will yield the desired counterexample.

Because $M$ is countable in $H[x]$, it has a real code. Let $\mu,\varepsilon,\pi \in H$ be $\Q$-names so that $\<M^*, \in^{M^*}, \P^*> = \<\mu^x,\varepsilon^x,\pi^x>$ is an isomorphic copy of $\<M,\in^M,\P>$ on the natural numbers. Moreover, fix a condition $q \in x$ which forces this. Then \hbox{$\<M^*, \in^{M^*}, \P^*>$} is appropriate as input to $\Phi$. Because $\Phi$ is Borel as defined using a real parameter $y \in H$, we have that $H[x]$ can compute $\Phi(M^*, \in^{M^*}, \P^*)$. Moreover, in the forcing extension of $\<M^*, \in^{M^*}, \P^*>$ via the generic output by $\Phi$, call it $M^*[G^*]$, we can check the continuum pattern below $\aleph_\omega$ to determine which finite variant of $x$ was used in the lottery sum.

Take $p$ in this generic $G^*$ which forces that the finite variant chosen to code into the continuum pattern is $x \mathbin{\triangle} a$.  Now define $x'$ so that $x'(i) \ne x(i)$ for $i = \max(a \cup \set{\lvert p \rvert, \lvert q \rvert})+1$ and $x'$ agrees with $x$ on all other coordinates. Then $x'$ extends $q$ and is $H$-generic for $\Q$, and thus also $N$-generic. In particular, in $N[x']$ we have that $p$ forces the continuum pattern in the extension to be $x' \mathbin{\triangle} a$.
And since $x'$ and $x$ differ on a single coordinate, $M = N[x] = N[x']$ and $\P = \dot \P^x = \dot \P^{x'}$. Thus, $\<M^\dagger,\in^{M^\dagger}, \P^\dagger> = \<\mu^{x'}, \varepsilon^{x'}, \pi^{x'}>$ is isomorphic to $\<M^*,\in^{M^*},\P^*>$. Note now that $p \in G^\dagger = \Phi(M^\dagger,\in^{M^\dagger}, \P^\dagger)$ because $H[x] = H[x']$ and $x'$ was defined to agree with $x$ up to the amount of information needed about $\mu$, $\varepsilon$, and $\pi$ to have $\Phi$ put $p$ in the generic. But the models $M^*[G^*]$ and $M^\dagger[G^\dagger]$ have different continuum patterns, disagreeing at $\aleph_i$. 

The statements ``\GCH\ holds at $\aleph_n$'' or ``\GCH\ fails at $\aleph_n$'' show up in the theory of a model for standard natural numbers $n$. But if $M^*[G^*]$ and $M^\dagger[G^\dagger]$ are $\omega$-nonstandard and $i$ is nonstandard this may not be enough. This is where we use the assumption that the models extend a pointwise definable model of $V=L$; it follows that any nonstandard natural number in these models is definable. Thus, the statements show up in their theory, even if $n$ is nonstandard. Thus they cannot be elementarily equivalent, giving us the desired counterexample.
\end{proof}

Meanwhile, if we go to the projective level, then there will (consistently) be a functorial process. For example, if $V=L$, then in $L$ given any countable oracle code for a structure, we can find the $L$-least isomorphic copy in a projective way, specifically at the level $\Delta^1_2$. If we now build $M[G]$ using this copy, it will be constant on the isomorphism class of $M$, which is a very strong way of respecting isomorphism. Whether we can push this lower down in the projective hierarchy remains open.

\begin{question}
Is there an analytic (or co-analytic) functorial method to produce generic filters for models of set theory?
\end{question}

\printbibliography

\end{document}